\documentclass{article}
\usepackage[margin= 22mm,bottom=30mm, top= 21mm]{geometry}
\usepackage{amsthm}
\usepackage{amsmath}
\usepackage{amssymb}
\usepackage{setspace}
\usepackage{mathtools}
\usepackage{verbatim}
\usepackage{a4wide}
\usepackage{csquotes}
\usepackage[hidelinks]{hyperref}
\usepackage{cleveref}
\usepackage{enumitem}
\setlist[itemize]{leftmargin=*} 
\setlist[enumerate]{leftmargin=*}
\usepackage{framed}
\usepackage{floatrow}
\usepackage[T1]{fontenc}
\usepackage{bbm}
\usepackage{booktabs, tabularx, array}

\usepackage{mathdots}
\usepackage{xcolor}
\usepackage{diagbox}
\usepackage{colortbl}

\usepackage{graphicx}
\usepackage{subcaption}

\theoremstyle{plain}

\newtheorem{theorem}{Theorem}[section]
\newtheorem{claim}[theorem]{Claim}

\newtheorem{lemma}[theorem]{Lemma}
\newtheorem{corollary}[theorem]{Corollary}

\theoremstyle{definition}

\newtheorem{defn}[theorem]{Definition}
\newtheorem*{defn*}{Definition}
\newtheorem*{lem*}{Lemma}
\newtheorem*{claim*}{Claim}

\def\domagoj#1{{}{\textcolor{blue}{#1} }}

\expandafter\def\expandafter\normalsize\expandafter{%
    \normalsize
    \setlength\abovedisplayskip{4pt}
    \setlength\belowdisplayskip{4pt}
    \setlength\abovedisplayshortskip{4pt}
    \setlength\belowdisplayshortskip{4pt}
}

\usepackage[square,sort,comma,numbers]{natbib}
\def\domagoj#1{}

\let\domagoj=\domagojOpt 

\newcommand{\calG}{\mathcal{G}}
\newcommand{\calH}{\mathcal{H}}

\def\eps {\varepsilon}
\DeclareMathOperator{\E}{\mathbb{E}}

\newcommand{\cC}{\mathcal{C}}

\newcommand{\cE}{\mathcal{E}}
\newcommand{\cF}{\mathcal{F}}

\newcommand{\cH}{\mathcal{H}}

\newcommand{\bF}{\mathbb{F}}

\newcommand{\bx}{\mathbf{x}}
\newcommand{\by}{\mathbf{y}}

\newcommand{\norm}[1] {
    \left \| #1 \right \|
}

\renewcommand{\Pr}{\mathbb{P}}
\renewcommand{\E}{\mathbb{E}}

\newcommand{\floor}[1]{
    \left \lfloor #1 \right \rfloor
}

\makeatletter
\newcommand{\bigplus}{%
  \DOTSB\mathop{\mathpalette\mattos@bigplus\relax}\slimits@
}
\newcommand\mattos@bigplus[2]{%
  \vcenter{\hbox{%
    \sbox\z@{$#1\sum$}%
    \resizebox{!}{0.9\dimexpr\ht\z@+\dp\z@}{\raisebox{\depth}{$\m@th#1+$}}%
  }}%
  \vphantom{\sum}%
}
\makeatother

\title{On the asymptotics of the Erd\H os--Rogers function}
\author{Domagoj Brada\v{c}\thanks{Institute of Mathematics, EPFL, Lausanne, Switzerland. Email: \textbf{domagoj.bradac@epfl.ch}} \and Oliver Janzer\thanks{Institute of Mathematics, EPFL, Lausanne, Switzerland. Email: \textbf{oliver.janzer@epfl.ch}} \and Rik Sarkar\thanks{Institute of Mathematics, EPFL, Lausanne, Switzerland. Email: \textbf{rik.sarkar@epfl.ch}}}
\date{}
\begin{document}

\maketitle

\begin{abstract}
    The Erd\H os--Rogers function $f_{\ell,s}(n)$ is the largest order of a $K_\ell$-free induced subgraph guaranteed to exist in every $K_s$-free graph on $n$ vertices. While this function is well understood for $s=\ell+1$, the case where $s$ is much larger than $\ell$ has remained wide open. A long-standing lower bound of Sudakov states that $f_{\ell,s}(n)\geq n^{\frac{\ell}{2s}+O_\ell(s^{-2})}$, while a recent result of Brada\v{c} shows that $f_{\ell,s}(n)\leq n^{\frac{\ell-1}{s-1}+o(1)}$.

    In this paper, we close this gap asymptotically by proving that $f_{\ell,s}(n)= n^{\frac{\ell}{2s}+O_\ell(s^{-2})}$. More precisely, we prove that for all $2\leq \ell<s$, we have $f_{\ell,s}(n)\leq n^{\frac{\ell}{2s-\ell}+o(1)}$. Our proof builds on Brada\v{c}'s recent tight construction for off-diagonal Ramsey numbers, which can be viewed as the $\ell=2$ case of our result.
\end{abstract}

\section{Introduction}

The Ramsey number $R(s,k)$ is the smallest $n$ such that every $n$-vertex graph contains a clique of size $s$ or an independent set of size $k$. The study of this function and its variants is one of the most prominent research areas in Combinatorics. In recent years, this topic has seen some remarkable breakthroughs, see, e.g. \cite{campos2026exponential,MattheusVerstraete2024,balister2024upper,campos2025new,ma2026exponential,hefty2025improving,off-diagonal}. We refer the reader to \cite{morris2026some} for a survey of these developments.

When $s$ is fixed and $k\rightarrow \infty$, this quantity is known as the \emph{off-diagonal Ramsey number}. Estimating this function is asymptotically equivalent to bounding the size of the largest independent set that every $K_s$-free $n$-vertex graph is guaranteed to contain. In 1962, Erd\H os and Rogers \cite{ErdosRogers1962} considered the following natural generalization of this problem: what is the largest $k$ such that every $n$-vertex $K_s$-free graph contains a $K_\ell$-free induced subgraph on $k$ vertices? This function, denoted by $f_{\ell,s}(n)$ has since become known as the \emph{Erd\H os--Rogers function}. Notice that $\ell=2$ recovers the (inverse of the) off-diagonal Ramsey number.

The study of this function has received tremendous attention over the last 60 years. The first result for $\ell>2$ was obtained by Erd\H os and Rogers \cite{ErdosRogers1962} themselves, who showed that $f_{\ell,\ell+1}(n)\leq n^{1-\eps_\ell}$ for some $\eps_\ell>0$. The first lower bound $f_{\ell,s}(n)\geq n^{1/(s-\ell+1)}$ was observed by Bollob\'as and Hind \cite{BollobasHind1991}. This, in particular, shows that $f_{\ell,\ell+1}(n)\geq n^{1/2}$.

Noting that $f_{\ell,s}(n)$ is monotone decreasing in $s$, the result of Erd\H os and Rogers implies that $f_{\ell,s}(n)\leq n^{1-\eps_\ell}$ holds for all $s>\ell$. This was considerably improved by Krivelevich \cite{Krivelevich1994,Krivelevich1995}.

\begin{theorem}[Krivelevich \cite{Krivelevich1995}] \label{thm:krivelevich}
    For every $3\leq \ell<s$, we have $f_{\ell,s}(n)\leq n^{\frac{\ell}{s+1}+o(1)}$.
\end{theorem}

The lower bound of Bollob\'as and Hind was drastically improved by Sudakov \cite{Sudakov2005New,Sudakov2005Large}, as one of the early applications of the dependent random choice method. The exact exponent in his lower bound is given by a recursive formula, and his bound is usually stated in the following asymptotic form.

\begin{theorem}[Sudakov \cite{Sudakov2005Large}] \label{thm:sudakov}
    For every $3\leq \ell<s$, we have $f_{\ell,s}(n)\geq n^{a_{\ell,s}+o(1)}$ for some $a_{\ell,s}=\frac{\ell}{2s}+O_\ell(s^{-2})$.
\end{theorem}

While Sudakov's lower bound had seen no polynomial improvement for any pair $(\ell,s)$, we can prove a modest polynomial improvement for pairs satisfying $\ell \ge 4$ and $s \ge \ell+3$ using a careful refinement of Sudakov's approach. We illustrate this refinement in Appendix~\ref{app:lower-bound},
where we present the argument for the case $s=\ell+3$, from which an improvement for all stated pairs follows by Sudakov's recursive bound.

Similarly, in the case where $s$ is large compared to $\ell$, Krivelevich's upper bound has not been improved until very recently. Our result will concern precisely this regime, but before we introduce it, we shall first discuss the complementary regime where $s$ is close to $\ell$, which has seen significant progress over the last 15 years.

The first major advance was achieved by Dudek and R\"odl \cite{DudekRodl2011}, who showed that $f_{\ell,\ell+1}(n)=O_{\ell}(n^{2/3})$, bounding the exponent away from 1. They introduced the influential technique of placing blow-ups of $K_\ell$ randomly in an algebraic object. Their methods were developed further by Wolfovitz \cite{Wolfovitz2013}, who showed that $f_{3,4}(n)\leq n^{1/2}(\log n)^{O(1)}$, and this was generalized to $f_{\ell,\ell+1}(n)\leq n^{1/2}(\log n)^{O_{\ell}(1)}$ by Dudek, Retter and R\" odl \cite{DudekRetterRodl2014}.

Moving on to $s>\ell+1$, Dudek, Retter and R\"odl asked whether $f_{\ell,\ell+2}(n)=o(n^{1/2})$, which was answered affirmatively by Gowers and Janzer \cite{GowersJanzer2020}. Their construction still uses random blow-ups of $K_\ell$ but places them completely at random rather than being constrained by an algebraic object. This gives better bounds at the cost of a more difficult analysis.

Subsequent improvements on the Erd\H os--Rogers function were inspired by various breakthroughs on off-diagonal Ramsey numbers (themselves, of course, being a special case of the Erd\H os--Rogers function). The first of these breakthroughs was achieved by Mattheus and Verstra\"ete \cite{MattheusVerstraete2024}, who showed that $R(4,k)=k^{3-o(1)}$, or, equivalently, $f_{2,4}(n)=n^{1/3+o(1)}$. Their approach was extended in two different directions: Janzer and Sudakov \cite{JanzerSudakov2025} showed that $f_{\ell,\ell+2}(n)\leq n^{\frac{2\ell-3}{4\ell-5}+o(1)}$, improving the earlier bound of Gowers and Janzer; and Mubayi and Verstra\"ete \cite{MubayiVerstraete2025} proved that $f_{\ell,\ell+1}(n)=O_\ell(n^{1/2}\log n)$, which improved the logarithmic term in the Dudek--Retter--R\"odl result. Very recently, Morris, Sahasrabudhe and Verstra\"ete \cite{MorrisSahasrabudheVerstraete2026} proved that $f_{\ell,\ell+1}(n)=O_\ell(\sqrt{n\log n})$. Their proof builds upon the recent breakthrough work of Hefty, Horn, King and Pfender~\cite{hefty2025improving} on the Ramsey number $R(3,k)$ and the random blow-up construction of Gowers and Janzer \cite{GowersJanzer2020}. Combined with the recent result of OpenAI \cite{OAI:A-Logarithmic-Independence-Bound-for-Clique-Free-Graphs-September-25-2026} on the independence number of $K_{\ell+1}$-free graphs for $\ell \ge 3$, which implies that $f_{\ell,\ell+1}(n)=\Omega_{\ell}(\sqrt{n \log n})$, we have $f_{\ell,\ell+1}(n)=\Theta_{\ell}(\sqrt{n \log n})$, determining $f_{\ell,\ell+1}(n)$ up to a constant factor depending only on $\ell$.

While our understanding of the Erd\H os--Rogers function $f_{\ell,s}$ has improved drastically in the regime where $s$ is close to $\ell$ over the last 15 years, progress on the case where $s$ is substantially larger than $\ell$ was very limited. A result of Gowers and Janzer \cite{GowersJanzer2020} improved Krivelevich's bound for $s\leq 2\ell-1$, but for $s\geq 2\ell$ and $\ell \geq 3$, Krivelevich's result has remained best. Recently, Brada\v{c} \cite{off-diagonal} determined the tight exponent for off-diagonal Ramsey numbers. Interestingly, his construction builds upon a construction of optimally pseudorandom $K_s$-free graphs due to Alon and Krivelevich~\cite{alon-krivelevich}, which was originally motivated by the problem of finding explicit constructions for the Erd\H{o}s--Rogers problem. In the equivalent language of the Erd\H os--Rogers function, Brada\v{c}'s result states that $f_{2,s}(n)=n^{\frac{1}{s-1}+o(1)}$. As pointed out in his paper, for general $\ell$ and $s$, this implies $f_{\ell,s}(n)\leq n^{\frac{\ell-1}{s-1}+o(1)}$, which is a modest improvement over Theorem \ref{thm:krivelevich} for $s\geq 2\ell$. Nevertheless, the gap between the exponent in this upper bound and in Sudakov's lower bound (Theorem~\ref{thm:sudakov}) remained large. In this paper, we close this gap asymptotically.

\begin{theorem} \label{thm:s large}
    For every $\ell\geq 2$ and $s>\ell$, there exists some $\alpha=\alpha(\ell,s)$ such that $$f_{\ell,s}(n)\leq n^{\frac{\ell}{2s-\ell}}(\log n)^{\alpha}.$$

    In particular, combined with Theorem \ref{thm:sudakov}, we have
    $$f_{\ell,s}(n)= n^{\frac{\ell}{2s}+O_{\ell}(s^{-2})}.$$
\end{theorem}

Before we discuss the proof method, let us make a few remarks about this bound. While the exponent has the right asymptotics for fixed $\ell$ and large $s$, it does not match the lower bound for any pair $(\ell,s)$ with $\ell\geq 3$.
On the other hand, it can be used to obtain major progress towards a question of Erd\H os \cite{Erdos1995}, asking whether it is true that for all $s>\ell+1$, we have $f_{\ell+1,s}(n)/f_{\ell,s}(n)\rightarrow \infty$ as $n\rightarrow \infty$. Prior to our work, an affirmative answer was known only in highly restricted parameter ranges. Our Theorem \ref{thm:s large} answers this question affirmatively for any $s$ that is sufficiently large compared to $\ell$. 

It is perhaps interesting to discuss the somewhat notorious $(\ell,s)=(3,5)$ case, which is the smallest case of the Erd\H os--Rogers problem for which the exponent is unknown. Sudakov's long-standing lower bound gives $f_{3,5}(n)\geq n^{5/12+o(1)}$, while the current best upper bound, due to Janzer and Sudakov, gives $f_{3,5}(n)\leq n^{3/7+o(1)}$. Interestingly, our Theorem \ref{thm:s large} matches this upper bound exactly, despite being a completely different construction.

In general, for $s=\ell+2$, the best upper bound remains $f_{\ell,\ell+2}(n)\leq n^{(2\ell-3)/(4\ell-5)+o(1)}$ due to Janzer and Sudakov; for $\ell+3\leq s\leq \frac{3\ell}{2}$, the result of Gowers and Janzer remains best, whereas our Theorem \ref{thm:s large} gives the best upper bound for all $s>\frac{3\ell}{2}$.

Let us briefly comment on our approach. The strategy in \cite{off-diagonal} is to bound the number of independent sets of a certain size in a suitable $K_s$-free graph, and then randomly sample its vertices so that, with positive probability, none of these independent sets survives. We follow the same general strategy, using the same $K_s$-free graph, but instead count $K_\ell$-free subsets of a suitable size and then apply a similar random sampling argument.

The counting argument in our setting is significantly harder and uses the hypergraph container method \cite{balogh2015independent,saxton2015hypergraph}. As in many applications of this method (see, for instance, \cite{morris2016number, chen2026maximum, balogh2019number,nie2024turan,jiang2022balanced,ferber2020supersaturated}), the main technical challenge is to establish a suitable balanced supersaturation result: we need to show that every sufficiently large set of vertices contains a large collection of $K_\ell$'s which are sufficiently well distributed. Proving this balanced supersaturation statement constitutes the main technical part of the paper.

The rest of the paper is organized as follows. In Section~\ref{sec: construction}, we describe the construction used to prove Theorem~\ref{thm:s large} and state some of its key properties. We also state our counting result for $K_\ell$-free subsets (Lemma~\ref{lem:counting-kl-free}) and deduce our main result from it. In Section~\ref{sec:overview}, we give an overview of the balanced supersaturation argument underlying the proof of Lemma~\ref{lem:counting-kl-free}. In Section~\ref{sec: preliminaries}, we establish the preliminary lemmas needed for the proof. In Section~\ref{sec: key lemma}, we prove our key technical lemma (Lemma~\ref{lem:key-technical-lemma}), which allows us to construct many cliques of odd order while controlling the common neighbourhoods of certain subsets of their vertices. In Section~\ref{sec: balanced-supersaturation}, we use the key technical lemma to establish the desired balanced supersaturation result for $K_\ell$'s. Finally, in Section~\ref{sec: hypergraph-containers}, we apply the hypergraph container method to complete the proof of Lemma~\ref{lem:counting-kl-free}. In Appendix \ref{app:lower-bound}, we prove the improved lower bound obtained by refining Sudakov's approach.

\section{The construction}
\label{sec: construction}

We prove Theorem \ref{thm:s large} by analysing the construction used in \cite{off-diagonal} to show that $R(s,k)=k^{s-1+o(1)}$. The basis of this construction is an optimally pseudorandom $K_s$-free graph due to Alon and Krivelevich~\cite{alon-krivelevich}, which we now describe. Let $t \ge 2$ and let $q$ be an arbitrary prime power. Let $PG(t, q)$ denote the finite geometry of dimension $t$ over the field $\bF_q$. The points of $PG(t, q)$ may be represented by the $1$-dimensional subspaces of $\bF_q^{t+1}$, or equivalently, by equivalence classes of nonzero vectors $(x_1, \dots, x_{t+1})$ of length $t+1$ over $\bF_q$ under the equivalence relation where two vectors are equivalent if one is a multiple of the other by a nonzero element of $\bF_q$. Let $G(t, q)$ be the graph with vertex set $PG(t, q)$ where two (not necessarily distinct) vertices represented by vectors $\bx$ and $\by$ are joined by an edge if and only if $\langle \bx, \by \rangle = \sum_{i=1}^{t+1} x_i y_i = 0$, i.e. if the corresponding $1$-dimensional subspaces are orthogonal. Note that whether $\langle \bx, \by \rangle = 0$ does not depend on the choice of the representatives, so the graph is well defined. Let us point out that $G(t, q)$ contains a loop at a vertex represented by $\bx$ if and only if $\langle \bx, \bx \rangle = 0$.

A graph, which may contain loops, is said to be an $(n, d, \lambda)$-graph if it has $n$ vertices, it is $d$-regular, and all but the largest eigenvalue (which equals $d$) of its adjacency matrix are at most $\lambda$ in absolute value. The following properties of $G(t, q)$ are shown in~\cite{alon-krivelevich}.

\begin{lemma}[\cite{alon-krivelevich}] \label{lem:properties}
    Let $t \ge 2$ be fixed. The graph $G(t, q)$ is an $(n, d, \lambda)$-graph, with $n = q^t (1 + o(1))$, $d = q^{t-1}(1 + o(1))$ and $\lambda = (1+o(1)) \sqrt{d}$, where the asymptotic notation is with respect to $q$ tending to infinity.
\end{lemma}

It was shown in~\cite{alon-krivelevich} that the subgraph of $G(t, q)$ induced on the set of non-self-orthogonal points is $K_{t+2}$-free. We now describe the construction of Brada\v{c} \cite{off-diagonal} which was used to obtain the tight exponent for the off-diagonal Ramsey problem.

\textbf{Construction.}
Let $t \ge 2$, let $q$ be a prime power and denote $G = G(t, q)$. Let $D^{*} = D^{*}(t, q)$ be the directed graph with vertex set $V(D^{*}) = \{ (a, b) \in V(G)^2 \mid ab \in E(G) \}$ and arc set $A(D^{*})$, where $((a,b),(a',b')) \in A(D^{*})$ if and only if $ab' \in E(G)$ and $a'b \not\in E(G)$. 

We say a digraph is \emph{$T_k$-free} if it does not contain a copy of the transitive tournament on $k$ vertices. Crucially, the digraph $D^{*}$ is $T_{t+1}$-free.
\begin{lemma}[{\cite[Lemma~2.11]{off-diagonal}}] \label{lem:t+1-free}
    For any $t \ge 2$ and $q$ a prime power, the digraph $D^{*}(t,q)$ is $T_{t+1}$-free.
\end{lemma}

The $\ell=2$ case of Theorem~\ref{thm:s large} was already proved in \cite{off-diagonal}. Therefore, throughout the remainder of the paper, fix integers $\ell$ and $s$ satisfying $3 \leq \ell <s$. Let $q$ be a sufficiently large prime power, let $G = G(s-1, q)$ and let $D^{*}=D^{*}(s-1,q)$ be the digraph constructed above. Given a permutation $\pi$ of $V(D^{*})$, define $\Gamma_{\pi} = \Gamma_{\pi}(s, q)$ to be the graph with vertex set $V(D^{*})$ and edge set
\[
E(\Gamma_{\pi}) =\bigl\{\{u,v\}: (u,v)\in A(D^{*})\ \text{and}\ \pi(u)  < \pi(v)\bigr\}.
\]
That is, $\Gamma_{\pi}$ contains an edge between two vertices if and only if they are joined by an arc of $D^{*}$ directed forward according to $\pi$. Since $D^{*}$ is $T_s$-free by Lemma~\ref{lem:t+1-free}, it follows that $\Gamma_{\pi}$ is $K_s$-free. 

Lemmas 2.3 and 2.12 in \cite{off-diagonal} imply the following. 

\begin{lemma}
    \label{lem:number-of-ind-sets}
    For sufficiently large $q$, there exists a permutation $\pi$ such that the number of independent sets of size $k=q(\log q)^3$ in $\Gamma_\pi$ is at most $\left(\frac{q^{s-1} \log q}{k} \right)^{k}$.
\end{lemma}

We fix a permutation $\pi$ satisfying the conclusion of Lemma \ref{lem:number-of-ind-sets}. For the remainder of the paper we write $\Gamma=\Gamma_{\pi}$ and $u < v$ whenever $\pi(u) < \pi(v)$. We now state our main lemma, from which Theorem \ref{thm:s large} can be easily deduced.

\begin{lemma} \label{lem:counting-kl-free}
    There exists $\beta=\beta(\ell,s)$ such that for all sufficiently large $q$, the number of $K_\ell$-free subsets in $\Gamma$ of size $m = q^{\ell/2}(\log q)^{\beta}$ is at most $\left(q^{s+\frac{\ell}{2}-3}\right)^{m}$. 
\end{lemma}

Before proving Theorem \ref{thm:s large}, let us briefly discuss the optimality of the bound in Lemma \ref{lem:counting-kl-free}. We first point out that $\Gamma$ is locally dense: every subset of $V(\Gamma)$ of size at least $q^{s-1+o(1)}$ has density at least $q^{-1-o(1)}$ (see Lemma~\ref{lem: local-density}). Moreover, since $\Gamma$ is $K_s$-free, the neighbourhood of any vertex is $K_{s-1}$-free. Using the local density of $\Gamma$, and repeatedly passing to neighbourhoods $s-\ell$ times, we obtain a $K_\ell$-free subset of $\Gamma$ of size at least $q^{-(s-\ell)-o(1)}|V(\Gamma)|=q^{s+\ell-3-o(1)}$, as $|V(\Gamma)|\ge q^{2s-3}/2$. Since $m=q^{\ell/2+o(1)}$, this gives at least $\binom{q^{s+\ell-3-o(1)}}{m} \ge \left( q^{s+\frac{\ell}{2}-3-o(1)}\right)^{m}$ $K_\ell$-free subsets of size $m$. Thus, up to the $q^{o(m)}$ factor, the bound in Lemma \ref{lem:counting-kl-free} is best possible. The choice of $m$ is also optimal; for much smaller values of $m$, one can show that there are close to $\binom{|V(\Gamma)|}{m}$ $K_\ell$-free subsets in $\Gamma$ of size $m$. Indeed, it is not hard to show that the number of copies of $K_\ell$ in $\Gamma$ is at most $|V(\Gamma)|^\ell q^{-\binom{\ell}{2}+o(1)}$. Hence, if $m = q^{\ell/2-\eps}$, for a constant $\eps > 0$, the expected number of copies of $K_\ell$ in a random subset of $V(\Gamma)$ of size $2m$ is at most $(2m)^{\ell}q^{-\binom{\ell}{2}+o(1)} \le m/2$. Therefore, at least half the subsets of $\Gamma$ of size $2m$ contain at most $m$ copies of $K_\ell$. Deleting one vertex from each such copy, it follows that at least half the subsets of size $2m$ contain a $K_{\ell}$-free subset of size $m$. Since each $K_{\ell}$-free subset is counted at most $\binom{|V(\Gamma)|-m}{m}$ times, we get at least 
\[
\frac{\binom{|V(\Gamma)|}{2m}}{2 \binom{|V(\Gamma)|-m}{m}}\ge \frac{1}{4^{m}}\binom{|V(\Gamma)|}{m}
\]
$K_\ell$-free subsets of $\Gamma$ of size $m$. We now deduce Theorem~\ref{thm:s large}.
\begin{proof}[Proof of Theorem \ref{thm:s large} given Lemma \ref{lem:counting-kl-free}]
Let $\beta$ be the constant given by Lemma \ref{lem:counting-kl-free} and set $m=q^{\ell/2}(\log q)^{\beta}$. Applying Lemma \ref{lem:counting-kl-free}, for all sufficiently large prime powers $q$, we obtain a $K_s$-free graph $\Gamma$ with $|V(\Gamma)| \ge q^{2s-3}/2$ such that the number of $K_\ell$-free subsets of $V(\Gamma)$ of size $m$ is at most $\left(q^{s+\frac{\ell}{2}-3}\right)^{m}$. Let $p=\left(q^{s+\frac{\ell}{2}-3} \right)^{-1}$. Clearly, $p \le 1$. Let $\Gamma^{*}$ be the subgraph of $\Gamma$ obtained by independently retaining each vertex of $\Gamma$ with probability $p$ and removing an arbitrary vertex from each $K_{\ell}$-free subset of size $m$. Since $\Gamma$ is $K_s$-free, so is $\Gamma^{*}$, and moreover, $\Gamma^{*}$ does not contain any $K_\ell$-free subset of size $m$. Finally,
\[
\mathbb{E}[|V(\Gamma^{*})|] \ge p|V(\Gamma)|-1 \ge \frac{q^{s-\frac{\ell}{2}}}{4}.
\]
Hence, for all sufficiently large prime powers $q$, there exists a $K_s$-free graph on at least $q^{s-\frac{\ell}{2}}/4$ vertices containing no $K_\ell$-free subset of size $m=q^{\ell/2}(\log q)^{\beta}$.
Using Bertrand's postulate, this implies Theorem \ref{thm:s large}.
    
\end{proof}

The rest of the paper will be devoted to proving Lemma \ref{lem:counting-kl-free}.

\medskip

\textbf{Notation.}
We use standard graph theoretic notation. For a graph $H$, and vertex sets $A, B \subseteq V(H)$, we denote $e_{H}(A,B)=|\{(a,b) \in A \times B \mid ab \in E(H)\}|$. For $U \subseteq V(H)$, write $e_{H}(U)=|E(H[U])|$.

Recall that the vertices of $\Gamma$ are ordered pairs $(a,b)\in V(G)^2$. For $S\subseteq V(\Gamma)$, let
\[
L(S):=\{a\in V(G):(a,b)\in S\text{ for some }b\in V(G)\},
\qquad
R(S):=\{b\in V(G):(a,b)\in S\text{ for some }a\in V(G)\}
\]
be the left and right projections of $S$, respectively.
For $a,b\in V(G)$, define 
\[
m_S^L(a):=|\{b'\in V(G):(a,b')\in S\}|,
\qquad
m_S^R(b):=|\{a'\in V(G):(a',b)\in S\}|.
\]
We refer to $m_S^L(a)$ and $m_S^R(b)$ as the left and right multiplicities, respectively.
Write $m_S^L:=\max_{a\in V(G)}m_S^L(a)$ and $m_S^R:=\max_{b\in V(G)}m_S^R(b)$.

For $v\in V(\Gamma)$, let $N^L(v):=\{u\in V(\Gamma):uv \in E(\Gamma), u<v\}$ and $N^R(v):=\{u\in V(\Gamma):uv \in E(\Gamma), v<u\}$ denote the left and right neighbourhoods of $v$, respectively.
For $A\subseteq V(\Gamma)$, write $N^L(A):=\bigcap_{v\in A}N^L(v)$ and $N^R(A):=\bigcap_{v\in A}N^R(v)$, with the convention that $N^L(\emptyset)=N^R(\emptyset)=V(\Gamma)$. 
For $A, B \subseteq V(\Gamma),$ we write $\overrightarrow{E}_\Gamma(A, B)$ for the set of ordered pairs $(a, b) \in A \times B$ such that $a < b$ and $ab \in E(\Gamma)$. We further denote $\overrightarrow{E}_\Gamma(A) = \overrightarrow{E}_\Gamma(A,A).$

Finally, for a hypergraph $\calH$ and $j\ge1$, we write $\Delta_j(\calH)$ for the maximum number of edges of $\calH$ containing any fixed set of $j$ vertices.

We use $\log x$ to denote the natural logarithm of $x$. We systematically ignore floor and ceiling signs whenever they are not crucial to the argument.

\section{Proof overview} \label{sec:overview}

We now give a detailed overview of the proof of Lemma~\ref{lem:counting-kl-free}. As discussed in the introduction, the lemma will be proved using the method of hypergraph containers, for which the main ingredient is a balanced supersaturation result for $K_\ell$'s in $\Gamma$. In other words, we seek a large collection of $K_\ell$'s which is sufficiently well distributed.

More precisely, to prove Lemma \ref{lem:counting-kl-free}, for every subset $Y \subseteq V(\Gamma)$ of size at least $q^{s+\ell-3+o(1)}$, we aim to construct an $\ell$-uniform hypergraph $\calH$ whose edges are $\ell$-cliques in $\Gamma[Y]$ and such that 
\[
e(\calH)\ge |Y|^\ell q^{-\binom{\ell}{2}-o(1)}
\qquad\text{and}\qquad
\Delta_j(\calH)\le
\left(|Y|q^{-\ell/2} \right)^{\ell-j}q^{o(1)}
\quad\text{for every }j\in[\ell].
\]
 
Recall that the vertices of $\Gamma$ are ordered pairs of vertices which form an edge in the underlying pseudorandom graph $G$. Recall also that for a set $S \subseteq V(\Gamma)$, $L(S)$ and $R(S)$ denote its left and right projections, respectively. We say informally that a set $S$ is \emph{nearly regular} if every element of $L(S)$ appears as the left coordinate of roughly the same number of vertices in $S$, and similarly, every element of $R(S)$ appears as the right coordinate of roughly the same number of vertices in $S$. 

If two vertices $(x,y)$ and $(u,v)$ are adjacent in $\Gamma$, then either $xv\in E(G)$ or $uy\in E(G)$. Thus, estimates for the number of edges between two nearly regular subsets of $V(\Gamma)$ can be reduced to estimates for the number of edges in $G$ between the left projection of one set and the right projection of the other, and vice versa. The pseudorandomness of $G$ gives useful estimates of this kind provided that the product of the sizes of the relevant projections is sufficiently large; more precisely, we require this product to be at least $q^s$.

We first prove that we may assume that $Y$ is nearly regular, with every element of $L(Y)$ appearing as the left coordinate of at most $\delta_L$ vertices of $Y$, and every element of $R(Y)$ appearing as the right coordinate of at most $\delta_R$ vertices of $Y$, where
\begin{equation} \label{eq: delta-product}
    \delta_L\delta_R\le q^{-1+o(1)}|Y|.
\end{equation}
This follows from a simple regularisation argument together with the pseudorandomness of $G$. In particular,
\[
    |L(Y)||R(Y)|
    \ge \frac{|Y|}{\delta_L}\cdot \frac{|Y|}{\delta_R}
    \ge q^{1-o(1)}|Y|.
\]
Thus, the left and right projections of $Y$ have large product, which will allow us to exploit the pseudorandomness of $G$ when constructing cliques in $\Gamma[Y]$.

Let us first sketch a simpler argument proving balanced supersaturation under the stronger assumption that $|Y|\ge q^{s+2\ell-4+o(1)}$. Already this weaker form of balanced supersaturation would imply $f_{\ell,s}(n) \le n^{\frac{\ell}{2s}+O_\ell(s^{-2})}$ when $\ell$ is fixed and $s \to \infty$. We construct an $\ell$-clique $v_1,\ldots,v_\ell$ one vertex at a time. There is a slight subtlety which we have suppressed thus far. In particular, an edge between the relevant projections in $G$ need not necessarily give rise to an edge in $\Gamma$, since the definition of $\Gamma$ also imposes a non-adjacency condition on the other pair of coordinates, and also a condition on the ordering. Thus, estimates for the number of edges in $G$ between the relevant projections only give upper bounds on the sizes of the corresponding common neighbourhoods. These estimates will be used to upper bound the codegrees, while in the actual proof we use the local density of $\Gamma$ separately to ensure that the set from which we choose the next vertex remains sufficiently large, thus allowing us to build many cliques. For the purposes of this overview, we ignore this distinction and assume that every set $K$ of at most $\ell-1$ vertices chosen so far satisfies
\begin{equation} \label{eq: common-neighbourhood-bound}
    |N_{\Gamma}(K)\cap Y|=|Y|q^{-|K|+o(1)}.
\end{equation}
Since the edge density of $\Gamma$ is roughly $1/q$, this is the expected size of the common neighbourhood of $K$ in $Y$. Suppose that $v_1,\dots,v_i$ have already been chosen, and fix $K\subseteq\{v_1,\dots,v_i\}$ with $|K|\le \ell-2$. Put
\[
    U_i=N_\Gamma(v_1,\dots,v_i)\cap Y
    \qquad\text{and}\qquad
    W=N_\Gamma(K)\cap Y.
\]
The next vertex $v_{i+1}$ must be chosen from $U_i$. To maintain the desired upper bound for the common neighbourhood of $K \cup \{v_{i+1}\}$, we want a typical vertex of $U_i$ to have at most $q^{-1+o(1)}|W|$ neighbours in $W$. We therefore want
$e_\Gamma(U_i,W)\le q^{-1+o(1)}|U_i||W|$. As discussed before, we obtain such an estimate by counting edges in $G$ between $L(W)$ and $R(U_i)$, and vice versa. 
For the pseudorandomness of $G$ to provide a sufficiently good bound on the number of edges between $L(W)$ and $R(U_i)$, we need
$|L(W)||R(U_i)|\ge q^s$. This inequality indeed holds, since equations~(\ref{eq: delta-product}) and (\ref{eq: common-neighbourhood-bound}) together imply
\begin{equation}
\label{equation: naive-projection-bound}
    |L(W)||R(U_i)|
    \ge \frac{|W|}{\delta_L}\cdot\frac{|U_i|}{\delta_R}
    \ge q^{-|K|-i+1-o(1)}|Y|
    \ge q^s,
\end{equation}
where we used $|W|=|Y|q^{-|K|+o(1)}$ and $|U_i|=|Y|q^{-i+o(1)}$, and the final inequality follows from $|K|\le\ell-2$, $i\le\ell-1$ and
$|Y|\ge q^{s+2\ell-4+o(1)}$. 

The pseudorandomness of $G$ then gives
$e_\Gamma(U_i,W)\le q^{-1+o(1)}|U_i||W|$. By averaging, almost all vertices in $U_i$ preserve the desired bounds on the common neighbourhoods. In this way, using the local density of $\Gamma$ to maintain a large candidate set for subsequent vertices of the clique, we build
$|Y|^\ell q^{-\binom{\ell}{2}-o(1)}$ $\ell$-cliques, and the upper bound on the common neighbourhoods implies that every fixed set of $j$ vertices is contained in at most
\begin{equation}
\label{equation: naive-codegree}
    \frac{|Y|^{\ell-j}}{q^{\binom{\ell}{2}-\binom{j}{2}}}\,q^{o(1)}
    =
    \left(\frac{|Y|}{q^{(\ell+j-1)/2}}\right)^{\ell-j}q^{o(1)}
    \le
    \left(\frac{|Y|}{q^{\ell/2}}\right)^{\ell-j}q^{o(1)}
\end{equation}
of the selected cliques, giving the desired balanced supersaturation result for $|Y|\geq q^{s+2\ell-4+o(1)}$.

We next explain how exploiting the ordering and additional structure of the common neighbourhoods improves this threshold from $q^{s+2\ell-4+o(1)}$ to the nearly-optimal threshold $q^{s+\ell-2+o(1)}$. Recalling that the vertices of $\Gamma$ are ordered, for a vertex $v\in V(\Gamma)$, write $N^L(v)$ and $N^R(v)$ for the set of neighbours of $v$ lying to its left and right, respectively. More generally, for a set $S\subseteq V(\Gamma)$, let
\[
    N^L(S):=\bigcap_{v\in S}N^L(v)
    \qquad\text{and}\qquad
    N^R(S):=\bigcap_{v\in S}N^R(v)
\]
denote their common left and right neighbourhoods. The simpler argument above does not make use of the ordering of the vertices of $\Gamma$. In the actual proof, this ordering plays a crucial role. We construct an ordered $\ell$-clique $\{v_1<\cdots<v_\ell\}$ from the two ends towards the middle, choosing the vertices in the order
\[
    v_1,v_\ell,v_2,v_{\ell-1},\ldots.
\]
Suppose we have chosen $i$ vertices $v_1,\dots,v_{i/2},
v_{\ell-i/2+1},\dots,v_\ell$, where $i$ is even for simplicity.
We choose the next vertex from
\[
U_i= N^R(\{v_1,\dots,v_{i/2}\})
\cap
N^L(\{v_{\ell-i/2+1},\dots,v_\ell\})
\cap Y.
\]
In equation~(\ref{equation: naive-projection-bound}), we used the lower bound $|R(U_i)|\ge |U_i|/\delta_R$, which follows from the fact that every element of $R(U_i)\subseteq R(Y)$ appears as the right coordinate of at most $\delta_R$ vertices of $Y$, and hence of $U_i$. However, this bound does not exploit the fact that $U_i$ is a common neighbourhood of several previously chosen vertices. The definition of $\Gamma$ suggests that heuristically, after imposing adjacency to the $ i/2 $ vertices already chosen from one side, a typical right coordinate should appear in at most $\delta_R q^{-i/2}$
vertices in $U_i$. Consequently, we expect the stronger lower bound
\begin{equation}
\label{equation: stronger-R(Ui)-bound}
    |R(U_i)|\ge \frac{|U_i|}{\delta_R q^{- i/2}}=\frac{|Y|q^{-i/2-o(1)}}{\delta_R}.
\end{equation}
Another key observation is that the codegree bounds obtained in equation~(\ref{equation: naive-codegree}) are stronger than what we actually need. This suggests that it is unnecessary to control the common neighbourhood of every subset of a clique. Instead, we only keep track of a carefully chosen collection of common neighbourhoods which is sufficient to obtain the desired codegree bounds. More precisely, a careful counting argument shows that it suffices to ensure that 
\[
    |N^R(I)\cap N^L(J)\cap Y|
    \le |Y|q^{-|I|-|J|+o(1)},
\]

where $I,J$ are subsets of the clique such that every vertex of $I$ precedes every vertex of $J$ in the ordering of $V(\Gamma)$, and $|I|,|J|\le \lceil \ell/2\rceil$. As in the simpler argument above, we assume for the sake of exposition that these neighbourhoods have roughly their expected size, and hence replace the inequality above by equality.

Suppose that $I,J$ are subsets of the vertices chosen so far, with every vertex of $I$ preceding every vertex of $J$, and put
\[
    W=N^R(I)\cap N^L(J)\cap Y.
\]
Suppose that we wish to add the next vertex $v\in U_i$ to $J$, where $|J|\le \lceil\ell/2\rceil-1$. In order to preserve the required common neighbourhood bound, we need to control the number of edges from $U_i$ to $W$ in the appropriate direction. The pseudorandomness of $G$ gives the required estimate provided the product $|L(W)||R(U_i)|$ is sufficiently large.

As before, the adjacency conditions imposed by the vertices of $I$ suggest that a typical left coordinate appears in at most $\delta_Lq^{-|I|}$ vertices of $W$. Consequently,
\begin{equation}
\label{equation: stronger-L(W)-bound}
    |L(W)|
    \ge
    \frac{|W|}{\delta_Lq^{-|I|}}
    =
    \frac{|Y|q^{-|J|-o(1)}}{\delta_L}.
\end{equation}
Together with equation~(\ref{equation: stronger-R(Ui)-bound}), this gives
\[
    |L(W)||R(U_i)|
    \ge
    \frac{|Y|^2q^{-i/2-|J|-o(1)}}{\delta_L\delta_R}
    \ge
    |Y|q^{1-i/2-|J|-o(1)}
    \ge
    |Y|q^{-(\ell-2)-o(1)},
\]
where we used equation~(\ref{eq: delta-product}), combined with the fact that $i\le \ell-1$ and $|J|\le \lceil\ell/2\rceil-1$. Hence, provided $|Y|\ge q^{s+\ell-2+o(1)}$, the product of the relevant projections is at least $q^s$, and the pseudorandomness of $G$ supplies the required estimate.

There is one additional issue hidden by this heuristic argument. In each common neighbourhood, there may be an exceptional set of vertices whose left or right coordinates appear substantially more often than expected. We keep track of these exceptional sets and use the pseudorandomness of $G$ once more to show that they are sufficiently small. This allows the preceding argument to be carried out for all the common neighbourhoods that we need to control. 

The preceding discussion captures the main ideas of the proof. Reaching the optimal threshold $|Y|\ge q^{s+\ell-3+o(1)}$ requires several further technical refinements. First, we can no longer control all the common neighbourhoods described above at their expected size: some of them may be larger by a factor of $q^{1/2}$. These weaker bounds are still sufficient to control all codegrees except the degrees of individual vertices. Second, the exceptional sets mentioned above need to be analysed more carefully. To reach the optimal threshold, we need to choose the vertices of the clique in pairs, proceeding from the two ends towards the middle, which allows us to obtain sufficiently strong bounds on these exceptional sets. Finally, there are some additional complications arising from the parity of $\ell$, and the argument turns out to be cleaner when $\ell$ is odd.

For these reasons, we first prove a slightly more robust balanced supersaturation statement for cliques of odd order. The additional information retained in this statement allows us both to control the degrees of individual vertices and to deduce the even case from the odd one. In Section~\ref{sec: key lemma}, we make the preceding heuristic argument rigorous by proving the key technical lemma, which provides the required control on the common neighbourhoods. In Section~\ref{sec: balanced-supersaturation}, we use the key lemma to prove a robust balanced supersaturation statement for cliques of odd order (Lemma~\ref{lem: odd-codegree-from-main-lemma}), and then deduce the even case from it (Lemma~\ref{lem:final-hypergraph}).

\section{Preliminary lemmas} \label{sec: preliminaries}
Recall that $G$ is an $(n, d, \lambda)$-graph with $n = (1+o(1))q^{s-1}, d = (1+o(1)) q^{s-2}$ and $\lambda = (1+o(1)) d^{1/2}.$ Throughout the rest of the paper, we shall assume that the parameters $\ell$ and $s$ are fixed and $q$ is sufficiently large. Specializing the weighted version of the expander mixing lemma (see e.g.~\cite[Lemma~8]{BrandonLund}) to $G,$ we obtain the following.
\begin{lemma} \label{lem:expander-mixing-weighted}
    For any $f, g \colon V(G) \rightarrow \mathbb{R}_{\ge 0}$, it holds that
    \[ \sum_{(x, y) \colon xy \in E(G)} f(x) g(y) \le \frac{2}{q} \norm{f}_1 \norm{g}_1 + 2 q^{s/2 - 1} \norm{f}_2 \norm{g}_2. \]
\end{lemma}

We also record the unweighted version.
\begin{corollary}
    \label{cor:expander-mixing-usual}
    Let $A,B \subseteq V(G)$. Then we have 
    \[
    e_G(A,B) \le 2|A||B|/q + 2q^{s/2 -1}\sqrt{|A||B|}.
    \]
\end{corollary}

\begin{proof}
    Apply Lemma \ref{lem:expander-mixing-weighted} with $f=\mathbbm{1}_A$ and $g=\mathbbm{1}_B$.
\end{proof}

A simple and useful corollary is the following.
\begin{corollary} \label{cor:weighted-large-neighbourhoods}
    For any $B \subseteq V(G)$ and any $\tau \ge 4 |B| / q, \tau > 0$, we have
    \[ \sum_{a \in V(G) \colon |N_G(a) \cap B| \ge \tau} |N_G(a) \cap B| \le \frac{16 q^{s-2} |B|}{\tau}. \]
\end{corollary}
\begin{proof}
    Let $A = \{ a \in V(G) \colon |N_G(a) \cap B| \ge \tau \}$ and note that we need to show that $e_G(A, B) \le 16q^{s-2} |B| / \tau.$ By definition, we have $e_G(A, B) \ge \tau |A|$. By Corollary~\ref{cor:expander-mixing-usual}, 
    \begin{equation} \label{eq:exp-mix-in-cor}
        e_G(A, B) \le 2|A||B| / q + 2q^{s/2-1} \sqrt{|A||B|}.
    \end{equation}
    Combining the two and using $\tau \ge 4|B| / q$ implies that $|A| \le 16 q^{s-2} |B| / \tau^2.$ Hence, 
    \[ |A| |B| \le 16 q^{s-2} |B|^2 / \tau^2 \le 4q^{s-1}|B| / \tau.\]
    Plugging this into~\eqref{eq:exp-mix-in-cor} yields 
    \[ e_G(A, B) \le \frac{2}{q} \cdot (4 q^{s-1} |B| / \tau) + 2q^{s/2-1} \sqrt{ 16 q^{s-2} |B|^2 / \tau^2} = 16 q^{s-2} |B| / \tau, \]
    which completes the proof.
\end{proof}

The following simple but crucial lemma uses the expander mixing lemma to bound the number of edges between two vertex sets of $\Gamma$ for which we have some control on the left or right multiplicities.
\begin{lemma} \label{lem:fibre-mixing}
    For any $A, B \subseteq V(\Gamma)$, 
    \[ |\overrightarrow{E}_{\Gamma}(A, B)| \le 2|A||B| / q + 2 q^{s/2-1} (m_A^L m_B^R |A| |B|)^{1/2}. \]
\end{lemma}
\begin{proof}
    Note that, by construction,
    \[  |\overrightarrow{E}_{\Gamma}(A, B)| \le \sum_{(x,y) \colon xy \in E(G)} m_A^L(x) m_B^R(y). \]
    Applying Lemma~\ref{lem:expander-mixing-weighted} with $f(x) = m_A^L(x)$ and $g(y) = m_B^R(y)$, we have
    \[ \sum_{(x,y) \colon xy \in E(G)} m_A^L(x) m_B^R(y) \le 2|A| |B| / q + 2 q^{s/2-1} \norm{f}_2 \norm{g}_2 \le 2|A| |B| / q + 2 q^{s/2-1} (m_A^L m_B^R |A| |B|)^{1/2}. \]
\end{proof}

The following lemma has two parts. Recall that we construct the desired $\ell$-cliques by picking two vertices at a time and we wish to control intersections of common neighbourhoods of various subsets of the produced $\ell$-cliques. For such subsets which only contain the leftmost of the two newly added vertices, the first part of the lemma will be used to argue that for a typical vertex $u$ in the current candidate set $U$, the intersection of a given set $W$ with the right neighbourhood of $u$ contains few vertices with unexpectedly large left multiplicity. Naturally, the analogous statement holds for subsets to which we add the rightmost of the two added new vertices, which we shall argue by symmetry.

The second part of the lemma will be used for subsets containing both of the new vertices by arguing that for a typical ordered edge $(u,v) \in \overrightarrow{E}_\Gamma(U)$, the intersection of the right neighbourhood of $u$, the left neighbourhood of $v$ and a given set $W$ contains few vertices with unexpectedly large left multiplicity. Again, the analogous symmetric statement holds, but we do not state it here as we shall argue by symmetry.

\begin{lemma} \label{lem:expected-number-of-tau-faulty-right-neighbourhood}
    Let $U, W \subseteq V(\Gamma)$ be arbitrary and let $\tau > 0$ be a real number satisfying $\tau \ge 4 m_W^L / q$. For $u \in U$, define $W_u = N^R(u) \cap W$ and $F_u = \{ (x, y) \in W_u \mid m_{W_u}^L(x) \ge \tau \}.$
    Then,
    \[ \sum_{u \in U} |F_u| \le \frac{16 q^{s-2} |W| m_U^L}{\tau}. \]
    Assume furthermore that $|N^R(u) \cap U| \le 100|U| / q$ for all $u \in U$. Then,
    \[ \sum_{(u, v) \in \overrightarrow{E}_{\Gamma}(U)} |F_u \cap N^L(v)| \le \left( |U||W| m_U^L q^{s-4} / \tau + |U| (m_U^L m_U^R |W|)^{1/2} q^{s-5/2} \right) \cdot (\log q)^2. \]
\end{lemma}
\begin{proof}
    Fix $x \in V(G)$ and let $R_x = \{ y \mid (x, y) \in W \}$. Since $\tau \ge 4m_W^L/ q \ge 4 |R_x| / q$, by Corollary~\ref{cor:weighted-large-neighbourhoods},
    \[ \sum_{a \in V(G) \colon |N_G(a) \cap R_x| \ge \tau} |N_G(a) \cap R_x| \le \frac{16 q^{s-2} |R_x|}{\tau}. \]    
    Recalling the definition of $m_U^L$, the number of triples $(u, x, y) \in U \times V(G) \times V(G)$ such that $(x, y) \in W_u$ and $m_{W_u}^L(x) \ge \tau$ is at most
    \[ m_U^L \cdot \sum_{x \in V(G)} \frac{16 q^{s-2} |R_x|}{\tau} \le \frac{16 m_U^L q^{s-2} |W|}{\tau}, \]
    which proves the first part of the lemma.

    Now assume that $|N^R(u) \cap U| \le 100 |U| / q$ for all $u \in U$. For $t \in [1, 2s \log q],$ let $F_u^t = \{ (x,y) \in W_u \mid m_{W_u}^L(x) \in [2^{t-1} \tau, 2^t \tau) \}$ and note that by the first part of the lemma, we have 
    \begin{equation} \label{eq:sum-Fut}
        \sum_{u \in U} |F_u^t| \le 16 q^{s-2} |W| m_U^L / (2^{t-1} \tau).
    \end{equation}
    By Lemma~\ref{lem:fibre-mixing}, for any $u \in U$ and $t \in [1, 2s \log q]$, we have
    \begin{align*}
        |\overrightarrow{E}_{\Gamma}(F_u^t, N^R(u) \cap U)| &\le 2 |F_u^t| |N^R(u) \cap U| / q + 2 q^{s/2-1} (2^t \tau \cdot m_U^R \cdot |F_u^t| \cdot |N^R(u) \cap U|)^{1/2}.
    \end{align*}
    Using Cauchy-Schwarz,~\eqref{eq:sum-Fut} and the assumption that $|N^R(u) \cap U| \le 100 |U| / q$ for all $u \in U,$ it follows that for any $t \in [1, 2s \log q]$,
    \begin{align*} 
        \sum_{u \in U} |\overrightarrow{E}_{\Gamma}(F_u^t, N^R(u) \cap U)| &\le 64 q^{-1} q^{s-2} |W| m_U^L / (2^t \tau) \cdot (100|U| / q)\\
        &+ 128|U| \cdot (|W| m_U^L m_U^R)^{1/2} \cdot q^{s-5/2}.
    \end{align*}
    Since for any $u \in U$, $F_u = \bigcup_{t=1}^{2s \log q} F_u^t$, summing over all $t \in [2s \log q]$ and using that $q$ is sufficiently large proves the second part of the claim.  
\end{proof}

The next two lemmas allow us to pass from an arbitrary set of vertices to a large subset with well-controlled left and right multiplicities. The first of these lemmas shows we can regularise the left and right multiplicities, losing only a polylogarithmic factor in the size of the set, and in the second lemma we use the pseudorandomness of the underlying graph $G$ to show that the product of the maximum left and right multiplicities cannot be too large.

\begin{lemma}\label{lem:coordinate-regularisation}
For every $X\subseteq V(\Gamma)$, there exist a subset $Y\subseteq X$ and positive reals $\delta_1,\delta_2$ such that $|Y|\ge |X|/(\log q)^3$, and
\[
\frac{\delta_1}{(\log q)^3}\le m_Y^L(a)\le\delta_1
\quad\text{for every }a\in L(Y),\qquad
\frac{\delta_2}{(\log q)^3}\le m_Y^R(b)\le\delta_2
\quad\text{for every }b\in R(Y).
\]
\end{lemma}

\begin{proof}
For integers $i,j\ge1$, define
\[
X_{i,j}=\{(a,b)\in X:2^{i-1}\le m_X^L(a)<2^i\text{ and }2^{j-1}\le m_X^R(b)<2^j\}.
\]
Since $|V(\Gamma)|=q^{O(1)}$, there are $O(\log q)$ possible values for each of $i$ and $j$. Hence, by the pigeonhole principle, there exist $i_0,j_0$ such that $|X_{i_0,j_0}|\ge |X|/O((\log q)^2)$. Set $X_0=X_{i_0,j_0}$, $\delta_1=2^{i_0}$ and $\delta_2=2^{j_0}$.

We now iteratively prune $X_0$. Let $T$ denote the current set, initially $T=X_0$. If there exists $a\in L(T)$ such that $m_T^L(a)<\delta_1/(\log q)^3$, delete all pairs of $T$ whose first coordinate is $a$. Similarly, if there exists $b\in R(T)$ such that $m_T^R(b)<\delta_2/(\log q)^3$, delete all pairs whose second coordinate is $b$. Repeat until no such coordinate remains, and denote the final set by $Y$.

We show that not too many pairs are deleted. Since every $a\in L(X_0)$ satisfies $m_X^L(a)\ge\delta_1/2$, we have $|L(X_0)|\le2|X|/\delta_1$, and hence fewer than $2|X|/(\log q)^3$ pairs are deleted because of first coordinates. Similarly, fewer than $2|X|/(\log q)^3$ pairs are deleted because of second coordinates. Thus fewer than $4|X|/(\log q)^3$ pairs are deleted in total. Since $|X_0|\ge |X|/O((\log q)^2)$, for sufficiently large $q$ we obtain $|Y|\ge |X|/(\log q)^3$.

By construction, every $a\in L(Y)$ and $b\in R(Y)$ satisfy $m_Y^L(a)\ge\delta_1/(\log q)^3$ and $m_Y^R(b)\ge\delta_2/(\log q)^3$. On the other hand, since $Y\subseteq X_0$, we have $m_Y^L(a)\le m_X^L(a)<\delta_1$ and $m_Y^R(b)\le m_X^R(b)<\delta_2$. The lemma follows.
\end{proof}

\begin{lemma}\label{product-deltas-bound}
Let $Y\subseteq V(\Gamma)$ satisfy $|Y| \ge q^{s-1}$, and suppose that there exist positive reals $\delta_1,\delta_2$ such that
\[
\frac{\delta_1}{(\log q)^3}\le m_Y^L(a)\le\delta_1
\quad\text{for every }a\in L(Y),\qquad
\frac{\delta_2}{(\log q)^3}\le m_Y^R(b)\le\delta_2
\quad\text{for every }b\in R(Y).
\]
Then we have
\[
\delta_1\delta_2\le \frac{|Y|(\log q)^7}{q}.
\]
\end{lemma}

\begin{proof}
Recall that $V(\Gamma)$ consists of ordered pairs of vertices forming an edge in $G$. Thus, we have
\[
|Y|\le e_G(L(Y),R(Y)) \le \frac{2}{q}|L(Y)||R(Y)|+2q^{s/2-1}\sqrt{|L(Y)||R(Y)|},
\]
where we used Corollary~\ref{cor:expander-mixing-usual}. Since \(m_Y^L(a)\ge \delta_1/(\log q)^3\) for every \(a\in L(Y)\), we have
$|L(Y)|\le |Y|(\log q)^3/\delta_1$ and similarly
$|R(Y)|\le |Y|(\log q)^3/\delta_2$. Hence
\[
|Y|\le \frac{2}{q}\frac{|Y|^2(\log q)^6}{\delta_1\delta_2}
+2q^{s/2-1}\frac{|Y|(\log q)^3}{\sqrt{\delta_1\delta_2}}.
\]
Suppose towards a contradiction that
$\delta_1\delta_2>|Y|(\log q)^7/q$. Then the right-hand side is
\[
O\left(\frac{|Y|}{\log q}\right)
+
O\left(|Y|\sqrt{\frac{q^{s-1}}{|Y|\log q}}\right)
=o(|Y|),
\]
where we used $|Y|\ge q^{s-1}$. This is a contradiction.
\end{proof}

\begin{corollary}\label{cor:regularised-subset}
Let $X\subseteq V(\Gamma)$ satisfy
$|X|\ge q^{s-1}(\log q)^3$. Then there exists $Y\subseteq X$ such that
\[
|Y|\ge \frac{|X|}{(\log q)^3}
\qquad\text{and}\qquad
m_Y^Lm_Y^R\le \frac{|Y|(\log q)^7}{q}.
\]
\end{corollary}

\begin{proof}
    The proof follows immediately from Lemmas \ref{lem:coordinate-regularisation} and \ref{product-deltas-bound}.
\end{proof}

Next, we establish a local density property of $\Gamma$. Using the bound on the number of independent sets given by Lemma \ref{lem:number-of-ind-sets}, we show that every sufficiently large subset of $V(\Gamma)$ must span many edges. This is proven implicitly in~\cite{off-diagonal}, but we shall deduce it here using Lemma~\ref{lem:number-of-ind-sets} as a black box.

\begin{lemma}
    \label{lem: local-density}
    Let $U\subseteq V(\Gamma)$ satisfy $|U|\ge 4q^{s-1}\log q$. Then
    \[
    e_\Gamma(U)\ge \frac{|U|^2}{q(\log q)^4}.
    \]
\end{lemma}

\begin{proof}
    Suppose, for a contradiction, that there exists $U\subseteq V(\Gamma)$ satisfying
    $|U|\ge 4q^{s-1}\log q$ and
    $e_\Gamma(U)\le |U|^2/(q(\log q)^4)$. Delete all vertices with degree at least
    $4|U|/(q(\log q)^4)$ in $\Gamma[U]$. There are at most $|U|/2$ such vertices, so we are left with a subset $U'\subseteq U$ with $|U'|\ge |U|/2$ and
    \[
    \Delta(\Gamma[U'])\le \frac{4|U|}{q(\log q)^4}\le \frac{8|U'|}{q(\log q)^4}.
    \]

    Let $k=q(\log q)^3$. We greedily build independent sets of size $k$ in $\Gamma[U']$. Having chosen $i<k$ vertices, there are at least
    \[
    |U'|-i(\Delta(\Gamma[U'])+1)
    \ge |U'|-\frac{8|U'|k}{q(\log q)^4}-k
    \ge \frac{|U'|}{2}
    \]
    choices for the next vertex, provided $q$ is sufficiently large. Therefore, $\Gamma[U']$ contains at least $(|U'|/2)^k$ ordered independent $k$-tuples, and hence at least
    \[
    \frac{1}{k!}\left(\frac{|U'|}{2}\right)^k
    >\left(\frac{|U'|}{2k}\right)^k
    \ge\left(\frac{|U|}{4k}\right)^k
    \ge \left( \frac{q^{s-1}\log q}{k}\right)^k
    \]
    independent sets of size $k$, where we used the assumption that $|U| \ge 4q^{s-1}\log q$. This contradicts the fact that the permutation $\pi$ was chosen to satisfy the conclusion of Lemma~\ref{lem:number-of-ind-sets}.
\end{proof}

We now prove the following consequence of Lemma \ref{lem: local-density}, which will be convenient for later applications.
\begin{lemma}
\label{lem: large-right/left-degree}
Let $U\subseteq V(\Gamma)$ satisfy $|U|\ge q^{s-1}(\log q)^{10}$. Then at least $|U|/(\log q)^{8}$ vertices $v\in U$ satisfy
\[
|U\cap N^R(v)|\ge \frac{|U|}{q(\log q)^8}.
\]
Similarly, at least $|U|/(\log q)^{8}$ vertices $v\in U$ satisfy
\[
|U\cap N^L(v)|\ge \frac{|U|}{q(\log q)^8}.
\]
\end{lemma}

\begin{proof}
Apply Lemma~\ref{lem:coordinate-regularisation} to $U$ to obtain
$Y\subseteq U$ and positive reals $\delta_1,\delta_2$. Then applying Lemma \ref{lem:fibre-mixing}, we get
\[
e_\Gamma(Y)=|\overrightarrow{E}_\Gamma(Y)| \le 2|Y|^2/q + 2q^{s/2-1}|Y|\sqrt{\delta_1 \delta_2}.
\]
Since $|Y| \ge |U|/(\log q)^3 \ge q^{s-1}(\log q)^7$, by Lemma \ref{product-deltas-bound}, we have $\delta_1\delta_2 \le |Y|(\log q)^7/q$. Combined with the lower bound on $|Y|$, this implies that the second term in the above inequality is at most $2|Y|^2/q$. Consequently, we have $e_\Gamma(Y) \le 4|Y|^2/q$.

Delete from $Y$ every vertex whose degree in $\Gamma[Y]$ is at least
$16|Y|/q$. Since
$e_\Gamma(Y)\le 4|Y|^2/q$, at most $|Y|/2$ vertices are deleted.
Thus we obtain a set $U_0\subseteq Y$ with $|U_0|\ge |Y|/2$ and
\[
\Delta(\Gamma[U_0])\le \frac{16|Y|}{q} \le \frac{32|U_0|}{q}.
\]
Moreover,
$|U_0|\ge |U|/(2(\log q)^3)\ge 4q^{s-1}\log q$ for sufficiently large $q$, so
Lemma~\ref{lem: local-density} gives
\[
\sum_{v\in U_0}|U_0\cap N^R(v)|
=e_\Gamma(U_0)
\ge \frac{|U_0|^2}{q(\log q)^4}.
\]

Let
\[
T:=\left\{v\in U_0:|U_0\cap N^R(v)|
\ge \frac{|U_0|}{2q(\log q)^4}\right\}.
\]
The vertices in $U_0\setminus T$ contribute at most
$|U_0|^2/(2q(\log q)^4)$ to the preceding sum, so the vertices in $T$
contribute at least this much. On the other hand,
\[
\sum_{v\in T}|U_0\cap N^R(v)|
\le |T|\frac{32|U_0|}{q},
\]
and hence $|T|\ge |U_0|/(2^6(\log q)^{4})$. Since
$|U_0|\ge |U|/(2(\log q)^3)$, for sufficiently large $q$ we have
$|T|\ge |U|/(\log q)^{8}$. Furthermore, every $v\in T$ satisfies
\[
|U\cap N^R(v)|
\ge |U_0\cap N^R(v)|
\ge \frac{|U_0|}{2q(\log q)^4}
\ge \frac{|U|}{q(\log q)^8}.
\]
The second part of the lemma follows identically.
\end{proof}

With the preliminary lemmas in place, we proceed to the proof of the key technical lemma of this paper.

\section{The key lemma}
\label{sec: key lemma}
\subsection{Setup}
In this section we fix an odd integer $r \le \ell$. We say a pair $(I,J)$, where $I, J \subseteq [r]$ is \emph{relevant} if $|I|, |J| \le \frac{r+1}{2}$ and $I = \emptyset$ or $J = \emptyset$ or $\max I + 1 < \min J$. Note that a relevant pair $(I, J)$ satisfies $i < j, \forall i \in I, j \in J$, which we abbreviate as $I < J$. We remark that relevant pairs $(I, J)$ with $J = \emptyset$ will not actually be needed for the proof but we include them to make the statements symmetric with respect to the ordering of $\Gamma$.

Given an index set $P$ of integers, we say a set $e \subseteq V(\Gamma)$ is \emph{$P$-labeled} to mean that its elements are indexed as $e = \{v_i \mid i \in P\}$, where $v_i < v_{i'}$ for any $i, i' \in P$ such that $i < i'$. For a subset $P' \subseteq P$, we define $e_{P'}=\{v_i \mid i \in P'\}$. We can now state the key lemma.
\begin{lemma} \label{lem:key-technical-lemma}
    There is a positive constant $C = C(\ell, s)$ such that the following holds. Let $r \le \ell$ be an odd integer. Let $S, T_1, T_2 \subseteq V(\Gamma)$ be given sets and suppose that $|S| \ge q^{s + r - 3} (\log q)^C$ and $T_1, T_2 \supseteq S$. Assume that $m_S^L m_S^R \le \frac{|S| (\log q)^{100}}{q}$ and for any $Z \in \{T_1, T_2\}$,  it holds that 
    \[ m_S^L m_Z^R \le \frac{|Z| (\log q)^{100}}{q} \text{ and } m_S^R m_Z^L \le \frac{|Z| (\log q)^{100}}{q}. \]

    Then, there is an $r$-uniform hypergraph $\calH$ on the vertex set $S$ such that $e(\calH) \ge |S|^r q^{-\binom{r}{2}} (\log q)^{-C}$ and every hyperedge of $\calH$ induces an $r$-clique in $\Gamma[S]$. Furthermore, each hyperedge $e$ is an $[r]$-labeled set $e = \{ v_i \mid i \in [r]\}$ (where $v_1 < v_2 < \dots < v_r$) such that for any relevant pair $(I, J)$, and any $Z \in \{T_1, T_2\}$, it holds that
    \begin{equation} \label{eq:neighbourhood-bounds}
        |N^R(e_I) \cap N^L(e_J) \cap Z| \le \begin{cases}
        |Z| q^{-|I| - |J| + 1/2} (\log q)^C, &\text{if } |I| = (r+1)/2 \text{ and } (r+1)/2 \in I;\\
        |Z| q^{-|I| - |J| + 1/2} (\log q)^C, &\text{if }|J| = (r+1)/2 \text{ and } (r+1)/2 \in J;\\
        |Z| q^{-|I|-|J|} (\log q)^C, &\text{otherwise.}
    \end{cases}
    \end{equation} 
\end{lemma}

Let us briefly discuss the statement of Lemma~\ref{lem:key-technical-lemma}. By Corollary~\ref{cor:regularised-subset}, it suffices to prove a balanced supersaturation statement for sets satisfying the multiplicity assumptions above. The sets $T_1$ and $T_2$ serve two purposes. Given a set $Y$ satisfying these assumptions, the condition~\eqref{eq:neighbourhood-bounds} with $Z=Y$ will provide the bounds needed to control all codegrees except the degrees of individual vertices. These will instead be controlled by a maximality argument. Roughly speaking, if a maximal collection of cliques satisfying the desired codegree bounds contains too few edges, then the set $Y'\subseteq Y$ of vertices of small degree is large. Applying Lemma~\ref{lem:key-technical-lemma} with $S=Y'$ and $T_1=T_2=Y$ then produces many further cliques supported on $Y'$, whose addition preserves the required codegree bounds by~\eqref{eq:neighbourhood-bounds}, contradicting maximality. Finally, allowing two possibly different sets $T_1$ and $T_2$ will be useful when deducing the case of even $\ell$ from the case of odd $\ell$. To understand the proof of Lemma~\ref{lem:key-technical-lemma}, we encourage the reader to think of the case $T_1=T_2=S$.

As mentioned in Section~\ref{sec:overview}, the strategy for proving Lemma~\ref{lem:key-technical-lemma} will be to construct ordered $r$-cliques from the ends towards the middle. More precisely, we pick $r-1$ vertices of the clique in pairs, with each newly chosen pair occupying the leftmost and rightmost available positions. This is done by induction in Lemma~\ref{lem:inductive-building}. In Lemma~\ref{lem:middle-vertex}, we choose the final vertex in the middle position.

Throughout the rest of this section, we fix $r, S, T_1, T_2$ satisfying the assumptions of Lemma~\ref{lem:key-technical-lemma}. Additionally, we choose constants $K_0, \dots, K_{(r+1)/2}$ and $C$ such that $\ell \ll K_0 \ll K_1 \ll K_2 \ll \dots \ll K_{(r+1)/2} \ll C$.

For $j \in [0, (r-1)/2],$ let $P_j = [j] \cup [r-j+1, r]$. Note that $P_0 = \emptyset$ and $P_{j+1}$ is obtained from $P_j$ by adding the smallest and largest element of $[r] \setminus P_j$. The set $P_j$ will represent the set of positions occupied by a $2j$-uniform hyperedge constructed after the first $j$ rounds.

Next, we make several crucial definitions encoding the inductive requirements for the hyperedges.
\begin{defn}[$(j; I, J, Z)$-nice] \label{def:nice}
    Let $e = \{v_i \mid i \in P_j\} \subseteq S$ be a $P_j$-labeled set of size $2j$. For a relevant pair $(I, J)$ and $Z \in \{T_1, T_2\}$, we say that $e$ is \emph{$(j; I, J, Z)$-nice} if
    \begin{equation} \label{eq:nice-small-Z}
        |N^R(e_{I \cap P_j}) \cap N^L(e_{J \cap P_j}) \cap Z| \le |Z| q^{-|I \cap P_j| - |J \cap P_j|} (\log q)^{K_j},
    \end{equation}
    and moreover, if $I \cup J \not\subseteq P_j$, then there is a partition $B \cup F = N^R(e_{I \cap P_j}) \cap N^L(e_{J \cap P_j}) \cap Z$, where $B = B(e, I, J, Z)$ and $F = F(e, I, J, Z)$, satisfying
    \begin{enumerate}[label=P\arabic*)]
        \item \label{item:partition-few-faulty} $|F| \le |Z| q^{-|I| - |J|} (\log q)^{K_j}$;
        \item \label{item:partition-left-multiplicity} if $J \not\subseteq P_j$, then $m_B^L \le m_Z^L q^{-|I \cap P_j|} (\log q)^{K_j}$;
        \item \label{item:partition-right-multiplicity} if $I \not\subseteq P_j$, then $m_B^R \le m_Z^R q^{-|J \cap P_j|} (\log q)^{K_j}$.
    \end{enumerate}
\end{defn}
We give some remarks about Definition~\ref{def:nice}. The condition~\eqref{eq:nice-small-Z} corresponds to the desired conclusion given in~\eqref{eq:neighbourhood-bounds}. Indeed, the right-hand side of~\eqref{eq:nice-small-Z} is, up to the polylogarithmic factor, the heuristic size of the common neighbourhood on the left, where $e_{I \cap P_j}, e_{J \cap P_j}$ are the vertices appearing in~\eqref{eq:neighbourhood-bounds} for the pair $(I, J)$ that have been chosen after the first $j$ rounds. If $I \cup J \subseteq P_j$, then~\eqref{eq:neighbourhood-bounds} has been established after the first $j$ rounds. 

Otherwise, to continue making progress, we would like to control the left and right multiplicities of the common neighbourhood $N^R(e_{I \cap P_j}) \cap N^L(e_{J \cap P_j}) \cap Z$. We cannot do this for the entire common neighbourhood, so we partition it into a bounded part $B$ and a faulty part $F$. Property~\ref{item:partition-few-faulty} will ensure that the size of the faulty part is insignificant compared to the desired size of the final neighbourhood $N^R(e_I) \cap N^L(e_J) \cap Z$, so we may safely ignore it. If we wish to control the intersection of $B$ with the left neighbourhood of a typical vertex $v$, we wish to have a bound on the left multiplicity of $B$, $m_B^L$. However, if $J \subseteq P_j$, we will not consider such an intersection in future steps, so we do not require any bound on $m_B^L$, which explains condition~\ref{item:partition-left-multiplicity}. The condition~\ref{item:partition-right-multiplicity} is, of course, symmetric.

\begin{defn}[candidate set, beautiful set]
    Let $e = \{v_i \mid i \in P_j\} \subseteq S$ be a $P_j$-labeled set of size $2j$. We say a set $U \subseteq S$ is a \emph{candidate set for $e$} if it satisfies the following.
    \begin{enumerate}[label=C\arabic*)]
        \item \label{item:U-in-the-middle} $U \subseteq N^R(e_{[j]}) \cap N^L(e_{\{ r-j+1,  \dots, r\}})$;
        \item \label{item:U-size} $|U| = |S| q^{-2j} (\log q)^{-K_j}$;
        \item \label{item:U-multiplicities} $m_U^L \le m_S^L q^{-j} (\log q)^{K_j}$ and, analogously, $m_U^R \le m_S^R q^{-j} (\log q)^{K_j}.$
    \end{enumerate}    
    We say $e$ is \emph{beautiful} if $e$ is $(j; I, J, Z)$-nice for every relevant pair $(I,J)$ and any $Z \in \{T_1, T_2\}$ and furthermore, there exists a candidate set for $e$.
\end{defn}
As the name suggests, a candidate set for a $P_j$-labeled set $e$ is a set from which we will choose subsequent vertices to extend $e$. Property~\ref{item:U-in-the-middle} corresponds to the fact that we will choose future vertices from the common neighbourhood of vertices in $e$ lying between the first $j$ and last $j$ vertices of $e$. Property~\ref{item:U-size} ensures we have enough vertices to choose from, and for convenience, we arbitrarily restrict it to an exact size. Property~\ref{item:U-multiplicities} will be crucial for controlling the desired common neighbourhoods as already discussed.

We shall prove the following lemma by induction on $j$.
\begin{lemma} \label{lem:inductive-building}
    For any $j \in [0, (r-1)/2]$, there exists a $2j$-uniform hypergraph $\calH$ with vertex set $S$, $e(\calH) \ge |S|^{2j} q^{-\binom{2j}{2}} (\log q)^{-K_j}$ such that every hyperedge of $\calH$ is a beautiful set inducing a $2j$-clique in $\Gamma[S]$.
\end{lemma}

In order to prove Lemma~\ref{lem:inductive-building}, we show that any beautiful edge of size $2j$ has many extensions. This will be done in two steps. The first step, given by Lemma~\ref{lem:good-candidate-pairs}, will be to prune the candidate set of a beautiful edge and construct many possible extensions with certain nice properties. The second step, given by Lemma~\ref{lem:prob-not-nice}, is to show that most of these extensions lead to a beautiful edge of size $2(j+1)$.

\subsection{Ensuring new candidate sets}
We start with a preparatory lemma. Given a beautiful set $e$ of size $2j \le r-3$, we prune its candidate set $U$ and construct $|U|^2 / q^{1 + o(1)}$ pairs $uv \in \overrightarrow{E}_\Gamma(U)$ from which we will choose the extensions of $e$.

\begin{lemma} \label{lem:good-candidate-pairs}
    Let $e$ be a beautiful $P_j$-labeled set, where $j \in [0, (r-3)/2]$, and let $U$ be a candidate set for $e$. Then, there exists a subset $U' \subseteq U$ and a set of ordered edges $\cE \subseteq \overrightarrow{E}_{\Gamma}(U')$ satisfying the following:
    \begin{enumerate}[label=G\arabic*)]
        \item \label{item:U'size} $|U'| \ge |U|/2$.
        \item \label{item:good-edge-degree} For any $v \in U'$, $|U' \cap N^R(v)|, |U' \cap N^L(v)| \le 32 |U'|/q$.
        \item \label{item:cE-large} $|\cE| \ge |U'|^2 / (q (\log q)^{100})$.
        \item \label{item:cE-min-degree} For any $uv \in \cE$, $|\{ u' \mid u'v \in \cE \}|,  |\{ v' \mid uv' \in \cE\}| \ge |U'| / (q (\log q)^{200})$.
        \item \label{item:good-edge-left-multiplicity} For any $uv \in \cE,$ writing $v = (v_1, v_2)$, it holds that $m^L_{U' \cap N^R(u)}(v_1) \le m^L_S (\log q)^{K_{j+1}/8} q^{-(j+1)}$;
        \item \label{item:good-edge-right-multiplicity} For any $uv \in \cE,$ writing $u = (u_1, u_2)$, it holds that $m^R_{U' \cap N^L(v)}(u_2) \le m^R_S (\log q)^{K_{j+1}/8} q^{-(j+1)}.$
        \item \label{item:good-edge-exists-candidate-set} For any $uv \in \cE,$ there exists a candidate set for $e \cup \{u, v\}$ contained in $U'$.
    \end{enumerate}
\end{lemma}
\begin{proof}
    Since $U$ is a candidate set for $e$, we have
    \begin{equation} \label{eq:size-of-U}   
        |U| = |S| q^{-2j} (\log q)^{-K_j} \ge q^s (\log q)^{C/2},
    \end{equation}
    where we used that $2j \le r-3$, $|S| \ge q^{s+r-3} (\log q)^C$ and that $C$ is a sufficiently large constant.
    By Lemma~\ref{lem:fibre-mixing}, we have
    \[ 
        e_\Gamma(U) \le \frac{2}{q} |U|^2 + 2 q^{s/2-1} \sqrt{|U| m_U^L} \sqrt{|U| m_U^R}.
    \]
    We claim that $q^{s/2-1} \sqrt{|U| m_U^L} \sqrt{|U| m_U^R} = o(|U|^2/q)$, or equivalently, $q^s m_U^L m_U^R = o(|U|^2)$. Note that 
    \[ q^s m_U^L m_U^R \le q^s (m_S^L q^{-j} (\log q)^{K_j}) (m_S^R q^{-j} (\log q)^{K_j}) \le |S| q^{s-2j-1 + o(1)}. \]
    On the other hand, $|U|^2 = |S|^2 q^{-4j + o(1)}$. Since $2j \le r  -3$ and $|S| \ge q^{s + r - 3 + o(1)}$, it follows that 
    \[ q^s m_U^L m_U^R \le |S| q^{s-2j-1 + o(1)} = o(|U|^2), \]
    as claimed. Hence,
    \begin{equation} \label{eq:e-gamma-upper-bound}
        e_\Gamma(U) \le 4 |U|^2 / q.
    \end{equation}

    Let $U' = \{ u \in U \mid |N_\Gamma(u) \cap U| \le 16|U| / q \}$, so clearly~\ref{item:good-edge-degree} is satisfied. Note that $|U'| \ge |U| - 2e_\Gamma(U) / (16|U| / q) \ge |U| / 2$ by \eqref{eq:e-gamma-upper-bound}, giving~\ref{item:U'size}. By Lemma~\ref{lem: local-density}, we have
    \begin{equation} \label{eq:U'-many-edges}
         e_\Gamma(U') \ge |U'|^2 / (q (\log q)^{4}).
    \end{equation}

    Let $(u, v)$ be a uniformly random ordered edge in $\overrightarrow{E}_{\Gamma}(U')$ and denote $N_{uv} = N^R(u) \cap N^L(v) \cap U'$. We will show that $(u, v)$ satisfies~\ref{item:good-edge-left-multiplicity}--\ref{item:good-edge-exists-candidate-set} with probability at least $(\log q)^{-30}$. Applying Lemma~\ref{lem: large-right/left-degree} twice, we have that 
    \[ \Pr[|N_{uv}| \ge |U'| / (q^2 (\log q)^{16})] \ge (\log q)^{-25}, \]
    where we used that $|U'| / (q (\log q)^8) \ge q^{s-1} (\log q)^{10}.$
    
    For $u \in U'$, denote $W_u = N^R(u) \cap U'$ and note that $|W_u| \le 32 |U| / q$. Let $\tau = m_S^L q^{-(j+1)} (\log q)^{K_{j+1}/8}$ and note that $\tau / m_U^L \ge (\log q)^{K_{j+1}/16} / q \ge 4/q$. Thus, applying Lemma~\ref{lem:expected-number-of-tau-faulty-right-neighbourhood} with both input sets equal to $U'$ and $\tau$, we obtain that the number of ordered edges $(u, (v_1, v_2)) \in \overrightarrow{E}_\Gamma(U')$ such that $m_{W_u}^L(v_1) \ge \tau$ is at most 
    \[ 16 q^{s-2} |U'| m_U^L / \tau \le 16 q^{s-1} |U'| (\log q)^{K_j -K_{j+1}/8} < |U'|^2 / (q(\log q)^{C/4}), \]
    where we used~\eqref{eq:size-of-U} and that $C$ is large compared to $K_{j+1}$. Hence, recalling~\eqref{eq:U'-many-edges}, we see that $(u, v)$ satisfies~\ref{item:good-edge-left-multiplicity} with probability at least
    \[ 1 -\frac{|U'|^2 / (q(\log q)^{C/4})}{e_\Gamma(U')} \ge 1 - (\log q)^{-30}, \] and completely symmetrically, $(u, v)$ satisfies~\ref{item:good-edge-right-multiplicity} with probability at least $1 - (\log q)^{-30}$.
    
    Let the random variable $X$ denote the number of vertices $(x, y) \in N_{uv}$ such that $m_{N_{uv}}^L(x) > \tau.$ Applying Lemma~\ref{lem:expected-number-of-tau-faulty-right-neighbourhood} with both input sets equal to $U'$, we get
    \begin{align*}
        \E[X] &\le e_\Gamma(U')^{-1} \cdot \left( |U'|^2 m_{U'}^L q^{s-4} / \tau + |U'|^{3/2} (m_{U'}^L m_{U'}^R)^{1/2} q^{s-5/2} \right) \cdot (\log q)^2
    \end{align*}
    We want to show that $\E[X] \le |S| q^{-2(j+1)} (\log q)^{-C/2}.$ Recall that $e_\Gamma(U') \ge |U'|^2 / (q (\log q)^4)$ and note that $m_{U'}^L \le m_U^L \le m_S^L q^{-j} (\log q)^{K_j}$. Thus, the first term satisfies
    \[ e_\Gamma(U')^{-1} \cdot \left( |U'|^2 m_U^L q^{s-4} / \tau\right) \cdot (\log q)^2 \le q^{s-2} (\log q)^{-K_{j+1}/16}. \]

    The second term satisfies
    \begin{align*} 
        e_\Gamma(U')^{-1} &\cdot \left( |U'|^{3/2} (m_{U'}^L m_{U'}^R)^{1/2} q^{s-5/2} \right) \cdot (\log q)^2 \le |U'|^{-1/2} (m_S^L m_S^R)^{1/2} q^{s-j-3/2} (\log q)^{6 + K_j}\\
        &\le (|S| q^{-2j} (\log q)^{-K_j}/2)^{-1/2} (|S| (\log q)^{100} / q)^{1/2} q^{s-j-3/2} (\log q)^{6 + K_j} \\
        &\le q^{s-2} (\log q)^{K_{j+1}/6},
    \end{align*}
    where in the last two inequalities we chose $K_{j+1}$ to be sufficiently large compared to $K_j$. Since $|S| \ge q^{s + r - 3} (\log q)^C$ for a large constant $C$ and $2j \le r-3$, we have
    \[ \E[X] \le |S| q^{-2(j+1)} (\log q)^{-C/2}, \]
    as desired. Completely symmetrically, let $Y$ denote the number of vertices $(x, y) \in N_{uv}$ such that $m_{N_{uv}}^R(y) > m_S^R q^{-(j+1)} (\log q)^{K_{j+1}/8}$. Then, we have $\E[Y] \le |S| q^{-2(j+1)} (\log q)^{-C/2}.$ 
    
    Taking $C$ sufficiently large, by Markov's inequality and a union bound, we see that with probability at least $(\log q)^{-26}$, it holds that $X, Y \le |S| q^{-2(j+1)} (\log q)^{-C/4}$, $|N_{uv}| \ge |U'| / (q^2 (\log q)^{16})$ and $(u,v)$ satisfy~\ref{item:good-edge-left-multiplicity}~and~\ref{item:good-edge-right-multiplicity}. In such a case, by removing from $N_{uv}$ all pairs $(x, y) \in N_{uv}$ for which $m_{N_{uv}}^L(x) > m_S^L q^{-(j+1)} (\log q)^{K_{j+1}/8}$ or $m_{N_{uv}}^R(y) > m_S^R q^{-(j+1)} (\log q)^{K_{j+1}/8}$ and then taking an arbitrary subset of size exactly $|S| q^{-2(j+1)} (\log q)^{-K_{j+1}},$ we obtain a candidate set for $e \cup \{u, v\}$. 
    
    We showed there exists a family $\cE$ satisfying \ref{item:good-edge-left-multiplicity}--\ref{item:good-edge-exists-candidate-set} of size $|\cE| \ge |\overrightarrow{E}_\Gamma(U')| (\log q)^{-26} > |U'|^2 / (q (\log q)^{30})$. Now, iteratively remove from $\cE$ all pairs $uv$ not satisfying~\ref{item:cE-min-degree} with respect to the current set $\cE$. Doing so, we have removed at most
    \[ 2 |U'| \cdot (|U'| / (q (\log q)^{200})) \le |\cE|/2 \] pairs in $\cE$, so the resulting set satisfies \ref{item:cE-large}--\ref{item:good-edge-exists-candidate-set}, as required.
\end{proof}

\subsection{Bounding the number of not nice extensions}
The following lemma shows that for a beautiful set $e$, a relevant pair $(I, J)$ and $Z \in \{T_1, T_2\}$, most of the ordered edges given by Lemma~\ref{lem:good-candidate-pairs} can be added to $e$ to give a $(j+1; I, J, Z)$-nice set.

\begin{lemma} \label{lem:prob-not-nice}
    Let $j \in [0, (r-3) / 2]$ and let $e$ be a beautiful $P_j$-labeled set with a candidate set $U$ and let $U' \subseteq U$ and $\cE \subseteq \overrightarrow{E_\Gamma}(U')$ satisfy~\ref{item:U'size}--\ref{item:good-edge-exists-candidate-set}. Let $(I, J)$ be a relevant pair and $Z \in \{T_1, T_2\}$. Let $uv \in \cE$ be chosen uniformly at random and let $e' = e \cup \{u,v\}$. Viewing $e'$ as a $P_{j+1}$-labeled set, the probability that $e'$ is not $(j+1; I, J, Z)$-nice is $o(1)$.
\end{lemma}
\begin{proof}
    Let $p_1 = j+1$ and $p_2 = r-j$, and note that these are the indices of the newly added vertices. If $(I \cup J) \cap \{ p_1, p_2 \} = \emptyset,$ the statement is trivial as for any $uv \in \cE$, $e \cup \{u,v\}$ is $(j+1; I, J, Z)$-nice by choosing $K_{j+1} \ge K_j$.

    So we may assume that $(I \cup J) \cap \{ p_1, p_2 \} \neq \emptyset$. By symmetry, we may further assume $I \cap \{p_1, p_2\} \neq \emptyset$. Note that
    \begin{equation} \label{eq:U'-big}
        |U'| \ge |U| / 2 \ge q^{s + r - 3 - 2j} (\log q)^{C/2} \ge q^s (\log q)^{C/2}.
    \end{equation} 
    
    Denote $W = N^R(e_{I \cap P_j}) \cap N^L(e_{J \cap P_j}) \cap Z$. Since $I \cup J \not\subseteq P_j$, there is a partition $B \cup F = W$ satisfying~\ref{item:partition-few-faulty}--\ref{item:partition-right-multiplicity} with respect to $e$. We have that 
    \begin{equation} \label{eq:size-of-B}
        |B| \le |W| \le |Z| q^{-|I \cap P_j| - |J \cap P_j|} (\log q)^{K_j}
    \end{equation}
    and
    \begin{equation} \label{eq:size-of-F}
        |F| \le |Z| q^{-|I| - |J|} (\log q)^{K_j}.
    \end{equation}

    Note that for any $a \in U'$, 
    \[ \Pr[u=a] \le |\cE|^{-1} \cdot |N^R(a) \cap U'| \le |U'|^{-1} (\log q)^{101}. \]
    Completely analogously, the same holds for $v,$ so we have
    \begin{equation} \label{eq:no-vtx-too-frequent}
        \Pr[u=a], \Pr[v=a] \le |U'|^{-1} (\log q)^{101}.
    \end{equation}
    
    Let us show that
    \begin{equation} \label{eq:EU'B}
        |\overrightarrow{E}_\Gamma(U', B)| \le |U'| |Z| q^{-|I \cap P_j| - |J \cap P_j| - 1} (\log q)^{K_{j+1}/4}.
    \end{equation}

    Since $I \not\subseteq P_j$, it holds that $m_B^R \le m_Z^R q^{-|J \cap P_j|} (\log q)^{K_j}$. By Lemma~\ref{lem:fibre-mixing}, we have
    \begin{equation} \label{eq:edges-to-B}
         | \overrightarrow{E}_\Gamma(U', B)| \le 2|U'| |B| / q + 2 q^{s/2-1} (m_{U'}^L m_{B}^R |U'| |B|)^{1/2}.
    \end{equation} 
    We have 
    \[ 2|U'||B| / q \le 2 |U'| |Z| q^{-|I \cap P_j| - |J \cap P_j| -1} (\log q)^{K_j} \le |U'| |Z| q^{-|I \cap P_j| - |J \cap P_j| - 1} (\log q)^{K_{j+1}/8}. \]
    To bound the second term, note that 
    \begin{equation} \label{eq:prod-mBL-mUR}
        m_{U'}^L m_{B}^R \le m_S^L q^{-j} (\log q)^{K_j} \cdot m_Z^R q^{-|J \cap P_j|} (\log q)^{K_j} \le |Z| q^{-|J \cap P_j| - j - 1} (\log q)^{2K_j + 100}.
    \end{equation}
    
    We claim that
    \begin{equation} \label{eq:prepare-spectral-term} 
        2 q^{s/2-1} (m_{U'}^L m_{B}^R |U'| |B|)^{1/2} \le |U'| |Z| q^{-|I \cap P_j| - |J \cap P_j|-1} (\log q)^{K_{j+1}/8}.
    \end{equation}
    Using \eqref{eq:size-of-B}, \eqref{eq:prod-mBL-mUR} and squaring, it is enough to show that
    \[ 4q^{s-j-1 + |I \cap P_j|} \le |U'| (\log q)^{K_{j+1}/4 - 3K_j -100}. \]
    Using that $|U'| \ge |S| q^{-2j} (\log q)^{-K_j}/2, |S| \ge q^{s+r-3} (\log q)^C$ and taking $C$ sufficiently large, it suffices to show that $s-j-1 + |I \cap P_j| \le s+r-3-2j$, or equivalently, $j + |I \cap P_j| \le r-2$. Since $j \le (r-3)/2$ and $|I \cap P_j| \le |I|-1 \le (r-1)/2,$ we have $j + |I \cap P_j| \le r-2,$ establishing~\eqref{eq:prepare-spectral-term}.

    Plugging back into~\eqref{eq:edges-to-B} proves~\eqref{eq:EU'B}. Denote 
    \begin{equation} \label{eq:gamma_1}
        \gamma_1 = |Z| q^{-|I \cap P_j| - |J \cap P_j| - 1} (\log q)^{K_{j+1}/2}.
    \end{equation}
    We call a vertex $x \in U'$ \emph{good} if $|N^R(x) \cap B| \le \gamma_1$. Recalling~\eqref{eq:no-vtx-too-frequent} and~\eqref{eq:EU'B}, we have
    \begin{equation} \label{eq:prob-large-neigh}
        \Pr[u \text{ is not good}], \Pr[v \text{ is not good}] \le \frac{|U'|^{-1} (\log q)^{101} |\overrightarrow{E}_\Gamma(U', B)|}{ \gamma_1} =  o(1),
    \end{equation}
    where we chose $K_{j+1}$ sufficiently large.

    Now, we consider several cases corresponding to different intersections of $\{p_1, p_2\}$ with the sets $I$ and $J$.
    
    \textbf{Case 1:} $|I \cap \{p_1, p_2\}| = 1, J \cap \{p_1, p_2\} = \emptyset.$\\
    If $p_1 \in I$, let $t = u$ and if $p_2 \in I,$ let $t = v$. In other words, $t$ is the new vertex whose position is in $I$.

    By~\eqref{eq:prob-large-neigh}, with probability $1-o(1)$, we have $|N^R(t) \cap B| \le \gamma_1,$ in which case $|N^R(t) \cap W| \le \gamma_1 + |F| \le  |Z| q^{-|I \cap P_{j+1}| - |J \cap P_{j+1}|} (\log q)^{3K_{j+1}/4}.$ Hence, $e'$ satisfies  \eqref{eq:nice-small-Z} with probability $1 - o(1)$.
    
    If $I \cup J \subseteq P_{j+1}$, there is nothing else to prove, so assume this is not the case. If $J \subseteq P_{j+1},$ then letting $B' = B \cap N^R(t)$ and $F' = (W \cap N^R(t)) \setminus B'$, we have $|F'| \le |F| \le |Z| q^{-|I|-|J|} (\log q)^{K_{j+1}}$ and $m_{B'}^R \le m_B^R \le m_Z^R q^{-|J \cap P_j|} (\log q)^{K_j} \le m_Z^R q^{-|J \cap P_{j+1}|} (\log q)^{K_{j+1}}$, implying that $e'$ is nice with probability $1 - o(1)$. 

    So assume that $J \not\subseteq P_{j+1}$. This implies that $p_1 \in I, p_2 \not\in I$ as otherwise $I < J$ cannot hold. Also, by definition, we have $m_B^L \le m_Z^L q^{-|I \cap P_j|} (\log q)^{K_j}.$ Let $B_u = N^R(u) \cap B$, $\tau = m_Z^L q^{-|I \cap P_{j+1}|} (\log q)^{K_{j+1}}$ and
    \[ F_u = \{ (x,y) \in B_u \mid m_{B_u}^L(x) \ge \tau\}. \]
    Setting $B' = B_u \setminus F_u$, we clearly have $m_{B'}^L \le \tau$ and as above, $m_{B'}^R \le m_B^R \le m_Z^R q^{-|J \cap P_{j+1}|} (\log q)^{K_{j+1}}$. Setting $F' = (F \cup F_u) \cap N^R(u)$, the partition $B' \cup F' = N^R(e'_{I \cap P_{j+1}}) \cap N^L(e'_{J \cap P_{j+1}}) \cap Z$, satisfies~\ref{item:partition-left-multiplicity} and~\ref{item:partition-right-multiplicity}. It remains to show that it satisfies~\ref{item:partition-few-faulty} with probability $1-o(1)$. 

    Note that $\tau > 4 m_B^L / q$, so applying Lemma~\ref{lem:expected-number-of-tau-faulty-right-neighbourhood} with $W = B$ and using \eqref{eq:no-vtx-too-frequent}, we have
    \[ \E[ |F_u|] \le (\log q)^{101} |U'|^{-1} \cdot 16 q^{s-2} |B| m_{U'}^L / \tau. \]
    We wish to show that $\E[|F_u|] = o(|Z| q^{-|I|-|J|} (\log q)^{K_{j+1}}).$ Note that $m_{U'}^L \le m_U^L \le m_S^L q^{-j} (\log q)^{K_j} \le m_Z^L q^{-j} (\log q)^{K_j}$, so $m_{U'}^L / \tau \le q^{|I \cap P_{j+1}| - j} (\log q)^{-K_{j+1}/2}$.
    
    Taking $C$ large enough and using~\eqref{eq:U'-big} and \eqref{eq:size-of-B}, it is enough to show 
    \[ q^{-s-r+3+2j} q^{s-2} \cdot |Z| q^{-|I \cap P_j| - |J \cap P_j|} q^{|I \cap P_{j+1}| - j} \le |Z| q^{-|I| - |J|}, \]
    which, using that $|I \cap P_{j+1}| = |I \cap P_j| + 1,$ reduces to showing that
    \[ j + |I| + |J| - |J \cap P_j| \le r-2. \]
    Recalling that $J \not\subseteq P_j, p_1 \in I$ and $\max I + 1 < \min J$, it follows that not all future positions can be covered by $I \cup J$, that is, $|(I \cup J) \setminus P_{j+1}| \le r - |P_{j+1}| - 1 = r - 2j - 3$. On the other hand, since $J \not\subseteq P_{j+1},$ we have $I \cap [r-j+1, r] = \emptyset$, which together with $r-j = p_2 \not\in I \cup J$, implies $|(I \cup J) \cap P_{j+1}| - |J \cap P_j| \le |I \cap P_{j+1}| \le j+1.$ Therefore,
    \[ j + |I| + |J| - |J \cap P_j| = j + |(I \cup J) \setminus P_{j+1}| + |(I \cup J) \cap P_{j+1}| - |J \cap P_j| \le j + (r-2j-3) + (j+1) = r-2, \]
    as needed. We conclude that $\E[|F_u|] = o(|Z| q^{-|I|-|J|} (\log q)^{K_{j+1}})$.

    Since $|F'| \le |F| + |F_u|,$ by Markov's inequality, it follows that the partition $B' \cup F'$ satisfies~\ref{item:partition-few-faulty} with probability $1 - o(1)$, which completes the proof of this case.

    \textbf{Case 2:} $p_1, p_2 \in I$.\\
    Let $U'' = \{ w \in U' \mid w \text{ is good and } \exists z \text{ such that } wz \in \cE\}$. For $w \in U''$, let $B_w = B \cap N^R(w)$ and $V_w = \{ z \in U' \mid wz \in \cE\}$. Note that for $w \in U''$, we have $|U'| / (q (\log q)^{200}) \le |V_w| \le |U' \cap N^R(w)| \le 32 |U'| / q$ by~\ref{item:good-edge-degree} and~\ref{item:cE-min-degree}, and $|B_w| \le \gamma_1$ since $w$ is good.
    
    By~\ref{item:good-edge-left-multiplicity}, we have $m_{V_w}^L \le m_S^L q^{-j-1} (\log q)^{K_{j+1}/8}$. Since $I \not\subseteq P_j$, we have $m_{B_w}^R \le m_B^R \le m_Z^R q^{-|J \cap P_j|} (\log q)^{K_j}.$ Applying Lemma~\ref{lem:fibre-mixing}, we have
    \begin{equation} \label{eq:EVxBx}
        |\overrightarrow{E_\Gamma}(V_w, B_w)| \le 2 |B_w| |V_w| / q + 2q^{s/2-1} (m_{V_w}^L m_{B_w}^R |V_w| |B_w| )^{1/2}.
    \end{equation}
    Let $\gamma = |Z| |U'| q^{-|I \cap P_j| - |J \cap P_j| - 3} (\log q)^{3 K_{j+1}/4}$. We aim to show that $|\overrightarrow{E_\Gamma}(V_w, B_w)| \le \gamma.$ For $w \in U''$, using that $|B_w| \le \gamma_1,$ because $w$ is good, we obtain    
    \begin{equation} \label{eq:VxBx}
        |V_w| |B_w| \le 32|U'|/q \cdot |Z| q^{-|I \cap P_j| - |J \cap P_j| - 1} (\log q)^{K_{j+1}/2} \le |U'| |Z| q^{-|I \cap P_j| - |J \cap P_j| - 2} (\log q)^{K_{j+1}/2+1},
    \end{equation}
    hence, $(2/q) |V_w| |B_w| \le \gamma / 2$.

    To bound the second term in~\eqref{eq:EVxBx}, using that $m_Z^R m_S^L \le |Z| (\log q)^{100} / q,$ we obtain
    \begin{equation} \label{eq:mVxmBx}
        m_{V_w}^L m_{B_w}^R \le m_S^L q^{-j-1} (\log q)^{K_{j+1}/8} \cdot m_Z^R q^{-|J \cap P_j|} (\log q)^{K_j} \le |Z| q^{-|J \cap P_j| - j - 2} (\log q)^{K_{j+1}/4}.
    \end{equation}

    Note that
    \[ 2q^{s/2-1} (m_{V_w}^L m_{B_w}^R |V_w| |B_w| )^{1/2}  \le \gamma / 2 \iff  q^{s-2} m_{V_w}^L m_{B_w}^R |V_w| |B_w| \le \gamma^2 / 16. \]
    Using~\eqref{eq:VxBx}~and~\eqref{eq:mVxmBx} and the definition of $\gamma$, it is enough to show that
    \[ q^{s-j + |I \cap P_j|} \le |U'|. \]
    Using~\eqref{eq:U'-big} and taking $C$ large enough, it suffices to verify that $j + |I \cap P_j| \le r-3$. Since $p_1, p_2 \in I$, we have $|I \cap P_j| \le |I| - 2 \le (r+1)/2 - 2 = (r-3)/2$. Therefore, $j + |I \cap P_j| \le (r-3)/2 + (r-3)/2 = r-3,$ as claimed. We conclude that $|\overrightarrow{E_\Gamma}(V_w, B_w)| \le \gamma.$ Hence, for any $w \in U''$,
    \[ \E[|B \cap N^R(u) \cap N^R(v)| \mid u=w] = \frac{|\overrightarrow{E_\Gamma}(V_w, B_w)| }{|V_w|} \le \gamma \cdot q (\log q)^{200} / |U'| \le |Z| q^{-|I \cap P_j| - |J \cap P_j| - 2} (\log q)^{7 K_{j+1}/8} .\]
    Note that $|I \cap P_{j+1}| + |J \cap P_{j+1}| = |I \cap P_j| + |J \cap P_j| + 2$. Denoting $\alpha = |Z| q^{-|I \cap P_{j+1}| - |J \cap P_{j+1}|} (\log q)^{K_{j+1}} $, by Markov's inequality, for any $w \in U'',$ 
    \[ \Pr[ |B \cap N^R(u) \cap N^R(v)| > \alpha/2 \mid u=w] = o(1). \]

    Since $u$ is good with probability $1 - o(1)$, we have $u \in U''$ with probability $1 - o(1)$ and it follows that $\Pr[ |B \cap N^R(u) \cap N^R(v)| > \alpha/2] = o(1)$. If $|B \cap N^R(u) \cap N^R(v)| \le \alpha/2$, then letting $B' = B \cap N^R(u) \cap N^R(v)$ and $F' = F \cap N^R(u) \cap N^R(v)$ provides the required decomposition of $Z \cap N^R(e'_{I \cap P_{j+1}}) \cap N^L(e'_{J \cap P_{j+1}})$. Indeed,~\eqref{eq:nice-small-Z} holds since $|Z \cap N^R(e'_{I \cap P_{j+1}}) \cap N^L(e'_{J \cap P_{j+1}})| \le |F| + |B'| \le \alpha/2 + |F| \le \alpha$. Furthermore, \ref{item:partition-few-faulty} holds since $|F'| \le |F|$, and~\ref{item:partition-right-multiplicity} follows from $|J \cap P_{j+1}| = |J \cap P_j|$ and ~\ref{item:partition-right-multiplicity} for $e$. Finally, note that $J \subseteq P_{j+1}$ since $p_2 \in I$, so the restriction in ~\ref{item:partition-left-multiplicity} does not apply. This concludes the proof of this case.

    \textbf{Case 3:} $p_1 \in I, p_2 \in J$.\\
    As in the previous case, let $U'' = \{ w \in U' \mid w \text{ is good and } \exists z \text{ such that } wz \in \cE\}$. For $w \in U'$, let $B_w = B \cap N^R(w)$ and $V_w = \{ z \in U' \mid wz \in \cE\}$. Note that for $w \in U''$, we have $|U'| / (q (\log q)^{200}) \le |V_w| \le |U' \cap N^R(w)| \le 32 |U'| / q$ by~\ref{item:good-edge-degree} and~\ref{item:cE-min-degree}, and $|B_w| \le \gamma_1$ since $w$ is good.
    
    Let $\tau = m_Z^L q^{-|I \cap P_j| - 1} (\log q)^{K_{j+1}/2}$ and for $w \in U'$, let $F_w = \{ (x, y) \in B_w \mid m_{B_w}^L(x) \ge \tau \}$ and $G_w = B_w \setminus F_w$. Let $\gamma_2 = |Z| q^{-|I \cap P_{j+1}| - |J \cap P_{j+1}|} (\log q)^{7 K_{j+1}/8}$. 

    Consider an arbitrary $w \in U''$. First we show that
    \begin{equation} \label{eq:prob-u-Gw}
        \Pr[|N^L(v) \cap G_w| \ge \gamma_2 \mid u = w] = o(1).
    \end{equation}
    
    By Lemma~\ref{lem:fibre-mixing}, it follows that
    \begin{equation} \label{eq:E-GwVw}
        |\overrightarrow{E}_\Gamma(G_w, V_w)| \le 2 |G_w| |V_w| / q + 2 q^{s/2-1} (m_{G_w}^L m_{V_w}^R |G_w| |V_w|)^{1/2}.
    \end{equation}
    Since $|G_w| \le |B_w| \le \gamma_1,$ we have
    \begin{equation} \label{eq:GwVw}
        |G_w| |V_w| \le \gamma_1 \cdot 32 |U'| / q \le |U'| |Z| q^{-|I \cap P_j| - |J \cap P_j| - 2} (\log q)^{3 K_{j+1}/4}.
    \end{equation}
    
    Therefore, setting $\gamma_3 = |U'| |Z| q^{-|I \cap P_j| - |J \cap P_j| - 3} (\log q)^{4 K_{j+1}/5}$, the main term satisfies $2 |G_w| |V_w| / q \le \gamma_3$. We claim that the same is true for the other term. To this end, note that $m_{V_w}^R \le m_{U'}^R \le m_S^R q^{-j} (\log q)^{K_j}$, and, by definition, $m_{G_w}^L \le \tau$. Recalling that $m_Z^L m_S^R \le (|Z| / q) \cdot (\log q)^{100}$, it follows that
    \begin{equation} \label{eq:mVwmGw}
        m_{G_w}^L m_{V_w}^R \le m_Z^L q^{-|I \cap P_j| - 1} (\log q)^{K_{j+1}/2} \cdot m_S^R q^{-j} (\log q)^{K_j} \le |Z| q^{-j - |I \cap P_j| - 2} (\log q)^{2K_{j+1}/3}.
    \end{equation}

    We wish to show that $2 q^{s/2-1} (m_{G_w}^L m_{V_w}^R |G_w| |V_w|)^{1/2} \le \gamma_3$. Squaring and using~\eqref{eq:GwVw} and~\eqref{eq:mVwmGw}, it is enough to show that
    \[ 4 q^{s-2} \cdot |Z| q^{-j - |I \cap P_j| - 2} (\log q)^{2K_{j+1}/3} \cdot  |U'| |Z| q^{-|I \cap P_j| - |J \cap P_j| - 2} (\log q)^{3 K_{j+1}/4} \le \gamma_3^2. \]
    Plugging in the definition of $\gamma_3$, it suffices to show that
    \[ q^{s-j+|J \cap P_j|} (\log q)^{K_{j+1}} \le |U'|. \]
    Using~\eqref{eq:U'-big} and taking $C$ large enough, it suffices to show that $s - j + |J \cap P_j| \le s$. Since $p_1 \in I,$ we have $J \cap P_j \subseteq [r-j+1,r]$ so $|J \cap P_j| \le j$, implying the inequality. We conclude that $|\overrightarrow{E}_\Gamma(G_w, V_w)| \le 2\gamma_3.$

    Hence, using Markov's inequality, we have
    \begin{align*} 
        \Pr[|N^L(v) \cap G_w| \ge \gamma_2 \mid u = w] &\le \frac{|\overrightarrow{E}_\Gamma(G_w, V_w)|}{|V_w| \gamma_2} \le \frac{2\gamma_3 q (\log q)^{200}}{\gamma_2 |U'|}\\
        &\le \frac{2|U'| |Z| q^{-|I \cap P_j| - |J \cap P_j| - 3} (\log q)^{4 K_{j+1}/5} \cdot (q (\log q)^{200})}{ |U'||Z| q^{-|I \cap P_{j+1}| - |J \cap P_{j+1}|} (\log q)^{7 K_{j+1}/8}} = o(1),
    \end{align*}
    where in the last inequality we used that $|I \cap P_{j+1}| + |J \cap P_{j+1}|=|I \cap P_j| + |J \cap P_j| + 2 $ and that $K_{j+1}$ is sufficiently large. This establishes~\eqref{eq:prob-u-Gw}.

    Denote $\gamma_F = |Z| q^{-|I| - |J|} (\log q)^{K_{j+1}}/16$. We claim that
    \begin{equation} \label{eq:prob-many-faulty}
        \Pr[|N^L(v) \cap F_u| \ge \gamma_F] = o(1).
    \end{equation} 

    By Markov's inequality,
    \[ \Pr[|N^L(v) \cap F_u| \ge \gamma_F] \le |\cE|^{-1} \gamma_F^{-1} \sum_{wz \in \overrightarrow{E}_\Gamma(U')} |N^L(z) \cap F_{w}|. \]

    Denote $\gamma_4 =|U'|^2 |Z| q^{-|I| - |J| - 1} (\log q)^{3K_{j+1}/4}$. Since $|\cE| \ge |U'|^2 / (q (\log q)^{100})$, we have $|\cE| \gamma_F \ge \gamma_4$, so it is enough to show that
    \begin{equation} \label{eq:needed-sum-faulty}
        \sum_{wz \in \overrightarrow{E}_\Gamma(U')} |N^L(z) \cap F_{w}| = o(\gamma_4).
    \end{equation}

    Since $J \not\subseteq P_j,$ we have $m_B^L \le m_Z^L q^{-|I \cap P_j|} (\log q)^{K_j}$. Note that $\tau \ge 4 m_B^L  / q$ since $K_{j+1}$ is large enough. Furthermore, recall that $|N^R(w) \cap U'| \le 32 |U'| / q$ for every $w \in U'$. Hence, by the second part of Lemma~\ref{lem:expected-number-of-tau-faulty-right-neighbourhood}, applied with $U = U'$ and $W = B$, we have
    \begin{equation} \label{eq:case3-double-mixing}
        \sum_{wz \in \overrightarrow{E}_\Gamma(U')} |N^L(z) \cap F_{w}| \le \left(|B||U'| m_{U'}^L q^{s-4} / \tau + |U'| (m_{U'}^L m_{U'}^R |B|)^{1/2} q^{s-5/2} \right) (\log q)^2
    \end{equation}

    Using $m_{U'}^L \le m_S^L q^{-j} (\log q)^{K_j}, m_S^L \le m_Z^L$,~\eqref{eq:size-of-B} and the definitions of $\tau$ and $\gamma_4$, we have
    \begin{align*}
        |B| |U'| m_{U'}^L q^{s-4} (\log q)^2 \tau^{-1} / \gamma_4 &\le q^{s + |I| + |J| - |J \cap P_j| - j - 2} / |U'|.
    \end{align*}

    To show the above is $o(1)$, using~\eqref{eq:U'-big} and taking $C$ large enough, it suffices to verify that $s + |I| + |J| - |J \cap P_j| - j - 2 \le s+r - 3 - 2j$, or equivalently, that 
    \begin{equation} \label{eq:case3-gaps}
        j + |I| + |J| - |J \cap P_j| \le r-1.
    \end{equation}
    Since $p_1 \in I,$ we have $J \cap P_j \subseteq [r-j+1,r]$ and since $p_2 \in J,$ we have $|I \cap P_j| \le j$. These imply $|(I \cup J) \cap P_j| \le |J \cap P_j| + j$. Furthermore, since $p_1 \in I, p_2 \in J$ and $\max I + 1 < \min J,$ we have $|(I \cup J) \setminus P_j| \le r - 2j - 1$. Hence,
    \[ j + |I| + |J| - |J \cap P_j| = j + |(I \cup J) \cap P_j| + |(I \cup J) \setminus P_j| - |J \cap P_j| \le 2j + r - 2j - 1 = r - 1, \]
    establishing~\eqref{eq:case3-gaps}.

    We turn to the second term in~\eqref{eq:case3-double-mixing}. Note that $m_{U'}^L m_{U'}^R \le m_S^L m_S^R q^{-2j} (\log q)^{2K_j} \le |S| q^{-2j-1} (\log q)^{2K_j + 100}.$ Hence, using ~\eqref{eq:size-of-B} and $|Z| \ge |S|$, we have
    \[ (|U'| (m_{U'}^L m_{U'}^R |B|)^{1/2} q^{s-5/2} (\log q)^2 / \gamma_4)^2 \le q^{2|I| + 2|J| + 2s - |I \cap P_j| - |J \cap P_j| - 2j - 4} / |U'|^2. \]
    To show the above is $o(1)$, using~\eqref{eq:U'-big} and taking $C$ sufficiently large, it suffices to show that
    \[  2|I| + 2|J| + 2s - |I \cap P_j| - |J \cap P_j| - 2j - 4 \le 2(s+r-3 - 2j), \]
    or equivalently, 
    \begin{equation} \label{eq:case3-second-term-gaps}
        2|I| + 2|J| - |I \cap P_j| - |J \cap P_j| + 2j \le 2r-2.
    \end{equation}
    Since $p_1 \in I, p_2 \in J$ and $\max I + 1 < \min J$, we have that $|(I \cup J) \setminus P_j| \le r-2j-1$, implying that $|I| + |J| - |I \cap P_j| - |J \cap P_j| \le r-2j-1$. Using additionally that $|I| + |J| \le r-1,$ inequality~\eqref{eq:case3-second-term-gaps} follows.

    We conclude that $\sum_{wz \in \overrightarrow{E}_\Gamma(U')} |N^L(z) \cap F_{w}| = o(\gamma_4)$, which as discussed, implies
    \[ \Pr[|N^L(v) \cap F_u| \ge \gamma_F] = o(1). \]
    We conclude that with probability $1 - o(1)$, it holds that 
    \begin{equation} \label{eq:case3-allproperties-left}
        |B \cap N^R(u)| \le \gamma_1, \; |N^L(v) \cap G_u| \le \gamma_2 \;\text{ and } |\;N^L(v) \cap F_u| \le \gamma_F.
    \end{equation}
    Denote $\tau' = m_Z^R q^{-|J \cap P_j| - 1} (\log q)^{K_{j+1}/2}$. For $w \in U'$, let $B_w' = N^L(w) \cap B$, $F_w' = \{ (x, y) \in B_w' \mid m_{B_w'}^R(y) \ge \tau' \}$ and $G_w' = B_w' \setminus F_w'$. Completely analogously, with probability $1 - o(1)$, it holds that 
    \begin{equation} \label{eq:case3-allproperties-right}
        |B \cap N^L(v)| \le \gamma_1, \; |N^R(u) \cap G_v'| \le \gamma_2 \; \text{ and } \; |N^R(u) \cap F_v'| \le \gamma_F.
    \end{equation}
    Suppose that~\eqref{eq:case3-allproperties-left} and~\eqref{eq:case3-allproperties-right} hold. Let $F' = (F \cap N^R(u) \cap N^L(v)) \cup (N^L(v) \cap F_u) \cup (N^R(u) \cap F_v')$ and $B' = (W \cap N^R(u) \cap N^L(v)) \setminus F'.$ Then, 
    $|F'| \le |F| + 2 \gamma_F \le |Z| q^{-|I| - |J|} (\log q)^{K_{j+1}} / 2$. Furthermore, note that $B' \subseteq G_u \cap N^L(v)$, so $|B'| \le \gamma_2$ and $m_{B'}^L \le \tau \le m_Z^L q^{-|I \cap P_{j+1}|} (\log q)^{K_{j+1}}$. Analogously, $m_{B'}^R \le \tau' \le m_Z^R q^{-|J \cap P_{j+1}|} (\log q)^{K_{j+1}}$. Finally, 
    \[ |W \cap N^R(u) \cap N^L(v)| \le |F'| + |B'| \le|Z| q^{-|I| - |J|} (\log q)^{K_{j+1}} / 2 + \gamma_2 \le |Z| q^{-|I \cap P_{j+1}| - |J \cap P_{j+1}|} (\log q)^{K_{j+1}}. \]
    Hence, with probability $1 - o(1)$, the edge $e \cup \{u, v\}$ is $(j+1; I, J, Z)$-nice as claimed.
\end{proof}

It is now easy to deduce Lemma~\ref{lem:inductive-building}.
\begin{proof}[Proof of Lemma~\ref{lem:inductive-building}]
    For $j=0$, the hypergraph with vertex set $S$ and containing only the single empty edge satisfies the statement. Now, suppose $0 \le j \le (r-3)/2$ and we have a $2j$-uniform hypergraph $\calH_j$ with vertex set $S$ and $e(\calH_j) \ge |S|^{2j} q^{-\binom{2j}{2}} (\log q)^{-K_j}$ in which every hyperedge is a beautiful set inducing a $2j$-clique in $\Gamma[S]$. 

    Fix $e \in E(\calH_j)$ and let $U$ be its candidate set. By Lemma~\ref{lem:good-candidate-pairs}, there exist $U' \subseteq U$ and a set of ordered edges $\cE \subseteq \overrightarrow{E}_\Gamma(U')$ satisfying \ref{item:U'size}--\ref{item:good-edge-exists-candidate-set}. Since the number of relevant pairs $(I, J)$ is a constant and there are two possibilities for $Z$, by Lemma~\ref{lem:prob-not-nice}, for at least $|\cE| / 2$ ordered edges $uv \in \cE,$ the edge $e \cup \{u, v\}$ forms a beautiful set. Include all such sets $e \cup \{u, v\}$ into $E(\calH_{j+1})$. Note that each edge $e \in E(\calH_j)$ extends to at least
    \[ \frac{1}{2} \big((|S| q^{-2j} (\log q)^{-K_j}) / 2\big)^2 / (q (\log q)^{100}) \ge |S|^2 q^{-4j-1} (\log q)^{-K_{j+1}/2} \]
    edges $e' \in E(\calH_{j+1})$. Since different edges $e \in E(\calH_j)$ produce distinct extensions, we have
    \[ e(\calH_{j+1}) \ge e(\calH_j) \cdot |S|^2 q^{-4j-1} (\log q)^{-K_{j+1}/2} \ge |S|^{2(j+1)} q^{-\binom{2(j+1)}{2}} (\log q)^{-K_{j+1}}. \]
    By construction, each hyperedge of $\calH_{j+1}$ is a beautiful set and induces a $2(j+1)$-clique in $\Gamma[S]$.
\end{proof}

\subsection{The middle vertex and completing the proof of the key lemma}
The final step is to choose the middle vertex. We have the following analogue of Lemma~\ref{lem:prob-not-nice} for the middle vertex.
\begin{lemma} \label{lem:middle-vertex}
    Let $e$ be a beautiful set of size $r-1$ with a candidate set $U$. Let $(I, J)$ be a relevant pair and let $Z \in \{T_1, T_2\}$. If $u \in U$ is chosen uniformly at random, then with probability $1 - o(1)$, the set $e \cup \{u\}$ satisfies~\eqref{eq:neighbourhood-bounds} with respect to $I, J$ and $Z$.
\end{lemma}
\begin{proof}
    Denote $j = (r-1) / 2$. If $j+1 \not\in I \cup J$, then~\eqref{eq:neighbourhood-bounds} holds for any $w \in U$ so there is nothing to prove. Let us assume that $j+1 \in I$, the case $j+1 \in J$ being completely analogous.

    Let $W = Z \cap N^R(e_{I \cap P_j}) \cap N^L(e_{J \cap P_j})$ and let $B \cup F = W$ be the partition satisfying~\ref{item:partition-few-faulty}--\ref{item:partition-right-multiplicity}. Let $\eps = 1/2$ if $|I| = (r+1)/2$ and $\eps = 0,$ otherwise, and let 
    \[ \gamma = |U| |Z| q^{-|I| - |J| + \eps} (\log q)^{K_{j+1}/2}. \]
    By Lemma~\ref{lem:fibre-mixing}, we have
    \begin{equation} \label{eq:middle-vtx-edges}
        |\overrightarrow{E}_\Gamma(U,B)| \le 2|U| |B| / q + 2q^{s/2-1} (m_U^L m_B^R |U| |B|)^{1/2}.
    \end{equation}
    Recall that $|U| = |S| q^{-2j} (\log q)^{-K_j} \ge q^{s-2} (\log q)^{C/2}$ and note that 
    \begin{equation} \label{eq:UB}
        |U||B| \le |U| \cdot |Z| q^{-|I \cap P_j| - |J \cap P_j|} (\log q)^{K_j} = |U| |Z| q^{-|I| - |J| + 1} (\log q)^{K_j},
    \end{equation}
    where we used that $|I \cap P_j| + |J \cap P_j| + 1 = |I| + |J|$. Hence, $2 |U| |B| / q \le \gamma$, by taking $K_{j+1}$ sufficiently large.
    
    Let us now show that the second term in~\eqref{eq:middle-vtx-edges} is at most $\gamma$ as well. Since $I \not\subseteq P_j$, we have $m_B^R \le m_Z^R q^{-|J \cap P_j|} (\log q)^{K_j}$ and recall that $m_U^L \le m_S^L q^{-j} (\log q)^{K_j}$. Hence,
    \[ m_U^L m_B^R \le m_S^L m_Z^R q^{-|J \cap P_j| - j} (\log q)^{K_{j+1}/8} \le |Z| q^{-|J \cap P_j| - j - 1} (\log q)^{K_{j+1}/6} =  |Z| q^{-|J| - j - 1} (\log q)^{K_{j+1}/6}. \] 
    Combining the above with~\eqref{eq:UB}, we obtain
    \begin{align*}
        \left( 2q^{s/2-1} (m_U^L m_B^R |U| |B|)^{1/2} \right)^2 / \gamma^2 \le q^{|I| + s - 2 - j - 2\eps} (\log q)^{K_{j+1}} / |U| \le q^{|I| - j - 2\eps} \le 1,
    \end{align*}
    where in the second to last inequality we took $C$ sufficiently large and in the last inequality we used that $|I| \le j + 2\eps$, which follows from the definition of $\eps$ and the fact that $|I| \le j+1$.

    Hence, $|\overrightarrow{E}_\Gamma(U,B)| \le 2 \gamma$. Therefore, by Markov's inequality, 
    \[ \Pr \left[|B \cap N^R(u)| > |Z| q^{-|I| - |J| + \eps} (\log q)^{K_{j+1}} / 2\right] \le 2\gamma / (|U| |Z| q^{-|I| - |J| + \eps} (\log q)^{K_{j+1}} / 2) = o(1). \]
    Recall that $|F| \le |Z| q^{-|I| - |J|} (\log q)^{K_j} <|Z| q^{-|I| - |J| + \eps} (\log q)^{K_{j+1}} / 2$ and note that $|W \cap N^R(u)| \le |B \cap N^R(u)| + |F|$. Hence, 
    \[ \Pr \left[|W \cap N^R(u)| > |Z| q^{-|I| - |J| + \eps} (\log q)^{K_{j+1}}\right] \le \Pr \left[|B \cap N^R(u)| > |Z| q^{-|I| - |J| + \eps} (\log q)^{K_{j+1}} / 2\right] = o(1), \]
    proving the lemma.
\end{proof}

We are ready to complete the proof of Lemma~\ref{lem:key-technical-lemma}.
\begin{proof}[Proof of Lemma~\ref{lem:key-technical-lemma}]
    By Lemma~\ref{lem:inductive-building}, there is an $(r-1)$-uniform hypergraph $\calH_{r-1}$ with vertex set $S$ and $e(\calH_{r-1}) \ge |S|^{r-1} q^{-\binom{r-1}{2}} (\log q)^{-K_{(r-1)/2}}$ such that every hyperedge is a beautiful set inducing an $(r-1)$-clique in $\Gamma[S]$.

    We construct the desired $r$-uniform hypergraph $\calH$ as follows. Consider an arbitrary edge $e \in E(\calH_{r-1})$ and let $U$ be its candidate set. Union bounding over the constantly many choices for a relevant pair $(I, J)$ and $Z \in \{T_1, T_2\}$, Lemma~\ref{lem:middle-vertex} implies that for at least $|U| / 2$ vertices $u \in U$, the set $e \cup \{ u \}$ satisfies~\eqref{eq:neighbourhood-bounds} for all relevant pairs $(I, J)$ and all $Z \in \{T_1, T_2\}$. We include all such sets $e \cup \{u\}$ into $E(\calH)$.

    By construction,~\eqref{eq:neighbourhood-bounds} is satisfied for all $e \in E(\calH)$, any relevant $(I, J)$ and $Z \in \{T_1, T_2\}$, so it remains to lower bound the number of edges of $\calH$. Recalling the lower bound~\ref{item:U-size} for a candidate set for a beautiful set of size $r-1$, we have
    \[ e(\calH) \ge e(\calH_{r-1}) \cdot (|S| q^{-(r-1)} (\log q)^{-K_{(r-1)/2}}) / 2 \ge |S|^r q^{-\binom{r}{2}} (\log q)^{-C}, \]
    for a sufficiently large constant $C$, as required.
\end{proof}

\section{Balanced supersaturation}
\label{sec: balanced-supersaturation}
In this section, we will use Lemma~\ref{lem:key-technical-lemma} to construct, in any sufficiently large subset of $V(\Gamma)$, a large collection of cliques with suitably controlled codegrees. We first establish this for cliques of odd order, in a slightly stronger form which also controls their common neighbourhoods inside an auxiliary set $T$. This additional control will be needed when we deduce the even case. Throughout this section and the next, all occurrences of $O(1)$ are understood to be $O_{\ell,s}(1)$. 

\begin{lemma}
\label{lem: odd-codegree-from-main-lemma}
There exists a constant $C_1=C_1(\ell,s)$ such that the following holds for all sufficiently large $q$. Let $3\le r\le \ell$ be an odd integer, and let $W,T\subseteq V(\Gamma)$ satisfy $|W|\ge q^{s+r-3}(\log q)^{C_1}$ and $T \supseteq W$. Furthermore, suppose that $ m_W^Lm_W^R\le |W|(\log q)^{20}/q,
m_W^Lm_T^R\le |T|(\log q)^{20}/q$ and
$m_W^Rm_T^L\le |T|(\log q)^{20}/q$.
Then there is an $r$-uniform hypergraph $\calH$ with vertex set $W$ such that every edge of $\calH$ is a clique in $\Gamma[W]$, and the following hold.
\begin{enumerate}[label=\textnormal{(A\arabic*)}]
    \item \label{item:many-hyperedges}
    $e(\calH)\ge |W|^rq^{-\binom{r}{2}}(\log q)^{-O(1)}$;
    \item \label{item:codegrees-inside-W} $\Delta_j(\calH)\le (|W|q^{-r/2})^{r-j}(\log q)^{O(1)}$ for every $j\in[r]$;
    \item \label{item:neighbourhood-in-T}for every edge $e=\{v_1<\cdots<v_r\}$ and every relevant pair $(I,J)$,
    \[
    |N^R(e_I)\cap N^L(e_J)\cap T|
    \le
    \begin{cases}
    |T|q^{-|I|-|J|+1/2}(\log q)^{O(1)},&
    \text{if }|I|=(r+1)/2\text{ and }(r+1)/2\in I,\\
    |T|q^{-|I|-|J|+1/2}(\log q)^{O(1)},&
    \text{if }|J|=(r+1)/2\text{ and }(r+1)/2\in J,\\
    |T|q^{-|I|-|J|}(\log q)^{O(1)},&
    \text{otherwise}.
    \end{cases}
    \]
\end{enumerate}

\end{lemma}

\begin{proof}
    Let $C = C(\ell, s)$ be given by Lemma~\ref{lem:key-technical-lemma} and set $C_1=C+1$. Pick a maximal family $\calH$ of ordered $r$-cliques in $\Gamma[W]$ such that the following hold.
    \begin{enumerate}[label=\textnormal{(B\arabic*)}]
        \item \label{item:single-degree}$\Delta_1(\calH)\le |W|^{r-1}q^{-\binom{r}{2}}$;
        \item \label{item:neighbourhood-in-T/W}for every edge $e=\{v_1<\cdots<v_r\}\in E(\calH)$, every relevant pair $(I,J)$, and every $Z\in\{W,T\}$,
        \[
        |N^R(e_I)\cap N^L(e_J)\cap Z|
        \le
        \begin{cases}
        |Z|q^{-|I|-|J|+1/2}(\log q)^C,
        &\text{if }|I|=(r+1)/2\text{ and }(r+1)/2\in I,\\
        |Z|q^{-|I|-|J|+1/2}(\log q)^C,
        &\text{if }|J|=(r+1)/2\text{ and }(r+1)/2\in J,\\
        |Z|q^{-|I|-|J|}(\log q)^C,
        &\text{otherwise}.
        \end{cases}
        \]
    \end{enumerate}
    Clearly, $\calH$ satisfies \ref{item:neighbourhood-in-T} and \ref{item:codegrees-inside-W} for $j=1$. We first prove \ref{item:many-hyperedges}.

    \begin{claim}
    \label{claim:edgecount1}
        We have $e(\calH)\ge |W|^rq^{-\binom{r}{2}}(\log q)^{-(C+1)}$.
    \end{claim}

    \begin{proof}
        Suppose towards a contradiction that
        $e(\calH)\le |W|^rq^{-\binom{r}{2}}(\log q)^{-(C+1)}$. Set
        $D=|W|^{r-1}q^{-\binom{r}{2}}$ and let
        \[
        W'=\{v\in W:d_{\calH}(v)\ge D/2\},\qquad W_0=W\setminus W'.
        \]
        Since $|W'|D/2\le\sum_{v\in W}d_{\calH}(v)=re(\calH)$, we have
        $|W'|\le 2r|W|(\log q)^{-(C+1)}<|W|/2$, and hence $|W_0|\ge |W|/2$.

        We apply Lemma~\ref{lem:key-technical-lemma} with $S=W_0$, $T_1=W$ and $T_2=T$. Since $W_0\subseteq W$,
        \[
        m_{W_0}^Lm_{W_0}^R\le m_W^Lm_W^R
        \le \frac{|W|(\log q)^{20}}q
        \le \frac{|W_0|(\log q)^{21}}q.
        \]
        Similarly,
        $m_{W_0}^Lm_W^R,m_{W_0}^Rm_W^L\le m_W^Lm_W^R\le |W|(\log q)^{20}/q$, while
        $m_{W_0}^Lm_T^R\le m_W^Lm_T^R\le |T|(\log q)^{20}/q$ and
        $m_{W_0}^Rm_T^L\le m_W^Rm_T^L\le |T|(\log q)^{20}/q$.
        Also $W,T \supseteq W_0$, and the required lower bound on $|W_0|$ follows from $|W_0|\ge |W|/2$ and the choice of $C_1$.

        Hence, Lemma~\ref{lem:key-technical-lemma} gives a family $\calH'$ of $r$-cliques in $\Gamma[W_0]$ satisfying \ref{item:neighbourhood-in-T/W} and
        \[
        e(\calH')\ge |W_0|^rq^{-\binom{r}{2}}(\log q)^{-C}
        \ge 2^{-r}|W|^rq^{-\binom{r}{2}}(\log q)^{-C}>e(\calH).
        \]
        Thus there is an edge $e\in E(\calH')\setminus E(\calH)$. Every vertex of $e$ lies in $W_0$ and therefore has degree less than $D/2$ in $\calH$, so adding $e$ preserves \ref{item:single-degree} for sufficiently large $q$. Property \ref{item:neighbourhood-in-T/W} holds for $e$ by construction, contradicting the maximality of $\calH$.
    \end{proof}

    We next prove the codegree bounds required in \ref{item:codegrees-inside-W} for $j \ge 2.$ We give a brief overview of the argument. After fixing the positions of the prescribed $j$ vertices, we choose the remaining vertices of the clique successively, working from the two ends towards the centre. At each step, the vertices already determined constrain the next vertex to lie in an appropriate common left and right neighbourhood, which we bound using \ref{item:neighbourhood-in-T/W}. When there are too many determined vertices on one side to apply \ref{item:neighbourhood-in-T/W} directly, we use only a suitable subset of them. We keep track of the loss arising from this truncation and show that, over all the steps, the total loss is small enough to give the desired codegree bound.

    \begin{claim}
    \label{claim:codegrees1}
    For every $2\le j\le r$, we have
    \[
        \Delta_j(\calH)\le \left(|W|q^{-r/2}\right)^{r-j}(\log q)^{O(1)}.
    \]
    \end{claim}
    
    \begin{proof}
        Fix $J=\{u_1<\cdots<u_j\}\subseteq W$ and a set of positions
        $Q=\{q_1<\cdots<q_j\}\subseteq[r]$. We bound the number of edges
        $e=\{v_1<\cdots<v_r\}\in E(\calH)$ such that $v_{q_t}=u_t$ for every
        $t\in[j]$. Since there are only $O_r(1)$ choices for $Q$, it suffices to prove the required bound for each fixed choice.
    
        We choose the vertices in the unfixed positions successively in the order
        \[
            v_1,v_r,v_2,v_{r-1},\ldots,
            v_{\lfloor r/2\rfloor},v_{\lceil r/2\rceil+1},
            v_{\lceil r/2\rceil},
        \]
        skipping the positions in $Q$. For $1\le i\le\lceil r/2\rceil$, let
        $Q_i=Q\cap\{i,\ldots,r-i+1\}$, and for every unfixed position
        $p\in[r]\setminus Q$, let $a_p$ denote the number of vertices of the clique that have already been determined when $v_p$ is chosen.
    
        Fix $i\in[\lfloor r/2\rfloor]$ and suppose first that $i\notin Q$. When choosing $v_i$, the vertices $v_1,\ldots,v_{i-1}$ to its left have already been chosen, while the vertices
        $\{v_k:k\in Q_i\}\cup\{v_{r-i+2},\ldots,v_r\}$
        to its right are already determined. Hence,
        $a_i=2(i-1)+|Q_i|$. Moreover,
        \[
            v_i\in
            N^R(\{v_1,\ldots,v_{i-1}\})\cap
            N^L\!\left(\{v_k:k\in Q_i\}\cup
            \{v_{r-i+2},\ldots,v_r\}\right)\cap W.
        \]
        If $i-1+|Q_i|\le\lfloor r/2\rfloor$, we apply \ref{item:neighbourhood-in-T/W} to these two sets of determined vertices and obtain at most
        \[
            |W|q^{-2(i-1)-|Q_i|}(\log q)^C
        \]
        choices for $v_i$. If $i-1+|Q_i|\ge\lceil r/2\rceil$, choose a subset of
        $\{v_k:k\in Q_i\}\cup\{v_{r-i+2},\ldots,v_r\}$ of size $\lceil r/2\rceil$ and apply \ref{item:neighbourhood-in-T/W} using this subset together with
        $\{v_1,\ldots,v_{i-1}\}$. In this case \ref{item:neighbourhood-in-T/W} gives at most
        \[
            |W|q^{-(i-1)-\lceil r/2\rceil+1/2}(\log q)^C
            =|W|q^{-(i-1)-r/2}(\log q)^C
        \]
        choices for $v_i$. Thus, in either case, the number of choices for $v_i$ is at most
        \[
            |W|q^{-(i-1)-\min\{r/2,\,i-1+|Q_i|\}}(\log q)^C
            =|W|q^{-a_i+b_i}(\log q)^C,
        \]
        where $b_i:=\max\{0,i-1+|Q_i|-r/2\}$.
    
        Similarly, fix $i\in[\lfloor r/2\rfloor]$ and suppose that $r-i+1\notin Q$. When choosing $v_{r-i+1}$, the vertices $v_{r-i+2},\ldots,v_r$ to its right have already been chosen, while the vertices
        $\{v_k:k\in Q_{i+1}\}\cup\{v_1,\ldots,v_i\}$
        to its left are already determined. Hence
        $a_{r-i+1}=2i-1+|Q_{i+1}|$, and
        \[
            v_{r-i+1}\in
            N^R\!\left(\{v_k:k\in Q_{i+1}\}\cup\{v_1,\ldots,v_i\}\right)
            \cap N^L(\{v_{r-i+2},\ldots,v_r\})\cap W.
        \]
        Applying \ref{item:neighbourhood-in-T/W} exactly as above, using all the determined vertices on the left if there are at most $\lfloor r/2\rfloor$ of them, and otherwise using any $\lceil r/2\rceil$ of them, the number of possible choices for $v_{r-i+1}$ is at most
        \[
            |W|q^{-(i-1)-\min\{r/2,\,i+|Q_{i+1}|\}}(\log q)^C
            =|W|q^{-a_{r-i+1}+b_{r-i+1}}(\log q)^C,
        \]
        where $b_{r-i+1}:=\max\{0,i+|Q_{i+1}|-r/2\}$.
    
        Finally, if the centre position $\lceil r/2\rceil$ is unfixed, then all the other $r-1$ vertices have already been determined when it is chosen, and
        \[
            v_{\lceil r/2\rceil}\in
            N^R(\{v_1,\ldots,v_{\lfloor r/2\rfloor}\})
            \cap
            N^L(\{v_{\lceil r/2\rceil+1},\ldots,v_r\})\cap W.
        \]
        Both sets appearing here have size $\lfloor r/2\rfloor$, so \ref{item:neighbourhood-in-T/W} gives at most
        $|W|q^{-(r-1)}(\log q)^C$ choices. Thus
        $a_{\lceil r/2\rceil}=r-1$, and we set
        $b_{\lceil r/2\rceil}=0$.
    
        Multiplying these bounds, the number of edges corresponding to the fixed choice of $Q$ is at most
        \[
            |W|^{r-j}q^{-\sum_{p\notin Q}a_p+\sum_{p\notin Q}b_p}(\log q)^{O(1)}.
        \]
        As the $r-j$ unfixed vertices are chosen successively, the values $a_p$ are
        $j,j+1,\ldots,r-1$. Hence
        $\sum_{p\notin Q}a_p=\binom r2-\binom j2$, and it suffices to show that
        \[
            \sum_{p\notin Q}b_p\le \binom{r}{2}-\binom{j}{2}-\frac{r(r-j)}{2}=\frac{(r-j)(j-1)}2.
        \]
    
        We bound the sum on the left-hand side by considering the pairs of positions $i$ and $r-i+1$. Suppose first that exactly one of these positions is unfixed. If $i\notin Q$ and $r-i+1 \in Q$, then
        \[
            |Q_i|\le\min\{j,r-2i+1\}\le\frac{j+r-2i+1}{2},
        \]
        and hence $b_i\le(j-1)/2$. If instead $r-i+1\notin Q$ and $i \in Q$, then
        \[
            |Q_{i+1}|\le\min\{j-1,r-2i\}\le\frac{j-1+r-2i}{2},
        \]
        and similarly $b_{r-i+1}\le(j-1)/2$.
    
        Suppose now that both positions are unfixed. Then $Q_i=Q_{i+1}$ and
        $|Q_{i}|\le\min\{j,r-2i\}$. Put
        $x=i-1+|Q_i|-r/2$. Then
        \[
            b_i+b_{r-i+1}=\max\{0,x\}+\max\{0,x+1\}.
        \]
        Since $r$ is odd, $x$ is a half-integer. If $x\le-1/2$, the right-hand side is at most $1/2\le j-1$, since $j \ge 2$. Otherwise $x\ge1/2$, and
        \[
            b_i+b_{r-i+1}
            =2i-1+2|Q_i|-r
            \le j-1,
        \]
        where we used $2|Q_i|\le j+r-2i$. Thus, each pair $\{i, r-i+1\}, i \in [\floor{r/2}]$ contributes at most $(j-1)/2$ times the number of its unfixed positions. Finally, the centre position contributes $0$ to the sum $\sum_{p \notin Q }b_p$. Since there are $r-j$ unfixed positions in total,
        \[
            \sum_{p\notin Q}b_p\le\frac{(r-j)(j-1)}2.
        \]
        The claim follows.
    \end{proof}
    Combining Claims \ref{claim:edgecount1} and \ref{claim:codegrees1}, the lemma follows.
\end{proof}

We now deduce the balanced supersaturation result for cliques of arbitrary order. The odd case follows directly from Lemma~\ref{lem: odd-codegree-from-main-lemma}. For even $\ell$, we reduce to the odd case by choosing a suitable first vertex $v_1$ and finding the remaining $\ell-1$ vertices inside its right neighbourhood.

\begin{lemma} \label{lem:final-hypergraph}
    There exists a constant $C_2=C_2(\ell,s)$ such that the following holds for all sufficiently large $q$. Let $Y \subseteq V(\Gamma)$ be such that $|Y|\ge q^{s+\ell-3}(\log q)^{C_2}.$ Suppose further that $m_Y^Lm_Y^R \le |Y|(\log q)^7/q$. Then, there exists an $\ell$-uniform hypergraph $\calH$ with vertex set $Y$ such that every hyperedge forms a clique in $\Gamma[Y]$. Moreover, we have $e(\calH) \ge |Y|^{\ell}q^{-\binom{\ell}{2}}(\log q)^{-O(1)}$, and for any $j \in [\ell]$, it holds that
    \[ \Delta_j(\calH) \le \left(\frac{|Y|}{q^{\ell/2}}\right)^{\ell-j} (\log  q)^{O(1)}. \]
    
\end{lemma}

\begin{proof}
    Let $C_1=C_1(\ell,s)$ be the constant given by Lemma~\ref{lem: odd-codegree-from-main-lemma}. For odd $\ell$, the lemma follows immediately from
    Lemma~\ref{lem: odd-codegree-from-main-lemma} applied with $W=T=Y$, $r=\ell$ and taking $C_2=C_1$.
    We therefore suppose that $\ell\ge4$ is even, and take $C_2=C_1+10$.

    \begin{claim}
    \label{many-good-v1}
    There exists a set $R\subseteq Y$ with $|R|\ge |Y|/(\log q)^9$ such that for every $v_1\in R$ there exists a set
    $Y_{v_1}\subseteq N^R(v_1)\cap Y$ satisfying
    \begin{enumerate}[label=\textnormal{(C\arabic*)}]
        \item \label{item:size-Yv1} $\displaystyle \frac{|Y|}{2q(\log q)^8} = |Y_{v_1}|$;
        \item \label{item:mult-Yv1} $\displaystyle m_{Y_{v_1}}^L\le \frac{4m_Y^L}{q}$ and
        $m_{Y_{v_1}}^R\le m_Y^R$.
    \end{enumerate}
    \end{claim}

    \begin{proof}
        For $v_1\in Y$, let $W_{v_1}=N^R(v_1)\cap Y$ and $F_{v_1}:=\left\{(x,y)\in W_{v_1}: m_{W_{v_1}}^L(x)>4m_Y^L/q\right\}$. By the first part of Lemma~\ref{lem:expected-number-of-tau-faulty-right-neighbourhood}, applied with
        $U=W=Y$ and $\tau=4m_Y^L/q$, we have
        \[
            \sum_{v_1\in Y}|F_{v_1}|\le 4q^{s-1}|Y|.
        \]
        Hence at most $|Y|/(\log q)^9$ vertices $v_1\in Y$ satisfy $|F_{v_1}|\ge 4q^{s-1}(\log q)^9$.
        On the other hand, Lemma~\ref{lem: large-right/left-degree} gives at least
        $|Y|/(\log q)^8$ vertices $v_1\in Y$ such that $|W_{v_1}|\ge \frac{|Y|}{q(\log q)^8}$.
        Thus, for sufficiently large $q$, there is a set $R\subseteq Y$ with
        $|R|\ge |Y|/(\log q)^9$ such that, for every $v_1\in R$,
        \[
            |W_{v_1}|\ge\frac{|Y|}{q(\log q)^8}
            \qquad\text{and}\qquad
            |F_{v_1}|\le 4q^{s-1}(\log q)^9.
        \]
        By the lower bound on $|Y|$, for sufficiently large $q$ we have $|F_{v_1}|\le \frac{|Y|}{2q(\log q)^8}$, and hence $|W_{v_1}\setminus F_{v_1}| \ge \frac{|Y|}{2q(\log q)^8}$.
        
        For each $v_1\in R$, choose
        $Y_{v_1}\subseteq W_{v_1}\setminus F_{v_1}$ with
        $|Y_{v_1}|=\frac{|Y|}{2q(\log q)^8}$. Then \ref{item:size-Yv1} holds, while the definition of $F_{v_1}$ gives
        $m_{Y_{v_1}}^L\le4m_Y^L/q$. Finally,
        $Y_{v_1}\subseteq Y$ gives
        $m_{Y_{v_1}}^R\le m_Y^R$, proving \ref{item:mult-Yv1}.
    \end{proof}

    Fix $v_1\in R$. We apply
    Lemma~\ref{lem: odd-codegree-from-main-lemma} with $r=\ell-1$,
    $W=Y_{v_1}$ and $T=Y$. By \ref{item:mult-Yv1} and the hypothesis on $Y$,
    \[
        m_{Y_{v_1}}^Lm_{Y_{v_1}}^R
        \le \frac{4m_Y^Lm_Y^R}{q}
        \le \frac{4|Y|(\log q)^7}{q^2}
        \le \frac{|Y_{v_1}|(\log q)^{20}}{q}.
    \]
    Moreover, we have $m_{Y_{v_1}}^Lm_Y^R, m_{Y_{v_1}}^Rm_Y^L \le m_Y^Lm_Y^R \le |Y|(\log q)^7/q$, where the last inequality holds by assumption. Finally, by \ref{item:size-Yv1} and the choice $C_2=C_1+10$,
    \[
        |Y_{v_1}|
        \ge \frac{|Y|}{2q(\log q)^8}
        \ge q^{s+\ell-4}(\log q)^{C_1}
    \]
    for sufficiently large $q$.

    Hence, Lemma~\ref{lem: odd-codegree-from-main-lemma} gives an
    $(\ell-1)$-uniform hypergraph $\calH_{v_1}$ on $Y_{v_1}$, every edge
    of which is an $(\ell-1)$-clique in $\Gamma[Y_{v_1}]$, satisfying
    \begin{enumerate}[label=\textnormal{(D\arabic*)}]
        \item \label{item:many-hyperedges-Yv1} $e(\calH_{v_1})\ge
        |Y_{v_1}|^{\ell-1}q^{-\binom{\ell-1}{2}}(\log q)^{-O(1)}$;
        \item \label{item:codegrees-in-Yv1} $\Delta_j(\calH_{v_1})\le
        (|Y_{v_1}|q^{-(\ell-1)/2})^{\ell-1-j}(\log q)^{O(1)}$
        for every $j\in[\ell-1]$;
        \item \label{item:neighbourhood-in-Y} for every edge
        $e=\{v_2<\cdots<v_\ell\}\in E(\calH_{v_1})$ and every
        $K\subseteq\{2,\ldots,\ell\}$ with $|K|\le\ell/2$,
        \[
        |N^L(\{v_i:i\in K\})\cap Y|
        \le
        \begin{cases}
        |Y|q^{-|K|+1/2}(\log q)^{O(1)},&
        \text{if }|K|=\ell/2\text{ and }\ell/2+1\in K,\\
        |Y|q^{-|K|}(\log q)^{O(1)},&
        \text{otherwise}.
        \end{cases}
        \]
    \end{enumerate}
    Here \ref{item:neighbourhood-in-Y} follows from \ref{item:neighbourhood-in-T} in
    Lemma~\ref{lem: odd-codegree-from-main-lemma} by taking $I=\emptyset$,
    after relabelling the positions of the $(\ell-1)$-clique as
    $2,\ldots,\ell$. We now define an $\ell$-uniform hypergraph $\calH$ with vertex set $Y$ by
    \[
        E(\calH):=
        \bigcup_{v_1\in R}
        \bigl\{\{v_1\}\cup e:e\in E(\calH_{v_1})\bigr\}.
    \]
    Since $Y_{v_1}\subseteq N^R(v_1)\cap Y$ and every edge of
    $\calH_{v_1}$ is an $(\ell-1)$-clique in $\Gamma[Y_{v_1}]$,
    every edge of $\calH$ is an $\ell$-clique in $\Gamma[Y]$. We will show that $\calH$ satisfies the conclusions of the lemma.

    \begin{claim}
    \label{claim:edgecount2}
        We have $e(\calH)\ge |Y|^\ell q^{-\binom{\ell}{2}}(\log q)^{-O(1)}$.
        
    \end{claim}

    \begin{proof}
        By \ref{item:many-hyperedges-Yv1} and the bounds on $|R|$ and $|Y_{v_1}|$,
        \[
            e(\calH)=\sum_{v_1\in R}e(\calH_{v_1})
            \ge \frac{|Y|}{(\log q)^9}
            \left(\frac{|Y|}{2q(\log q)^8}\right)^{\ell-1}
            q^{-\binom{\ell-1}{2}}(\log q)^{-O(1)}
            =
            |Y|^\ell q^{-\binom{\ell}{2}}(\log q)^{-O(1)}.
        \]
    \end{proof}

    \begin{claim}
    \label{claim:codegrees2}
    For every $j\in[\ell]$,
    \[
        \Delta_j(\calH)
        \le \left(|Y|q^{-\ell/2}\right)^{\ell-j}(\log q)^{O(1)}.
    \]
    \end{claim}
    
    \begin{proof}
        The claim clearly holds for $j=\ell$, so we fix $j\in[\ell-1]$ and let $J=\{u_1<\cdots<u_j\}\subseteq Y$.
        Fix a set of positions
        $Q=\{q_1<\cdots<q_j\}\subseteq[\ell]$. We bound the number of edges
        $e=\{v_1<\cdots<v_\ell\}\in E(\calH)$ satisfying
        $v_{q_t}=u_t$ for every $t\in[j]$. Since there are only
        $O_\ell(1)$ choices for $Q$, it suffices to prove the required bound
        for each fixed choice.
    
        Suppose first that $1\in Q$. Then $v_1$ is fixed. If $v_1\notin R$, there are no such edges, so we may assume that $v_1\in R$. If $j\ge2$,
        every such edge arises from an edge of $\calH_{v_1}$ containing the
        $j-1$ fixed vertices in the positions $Q\setminus\{1\}$. Hence,
        by \ref{item:codegrees-in-Yv1}, the number of such edges is at most
        \[
            \Delta_{j-1}(\calH_{v_1})
            \le
            \left(|Y_{v_1}|q^{-(\ell-1)/2}\right)^{\ell-j}
            (\log q)^{O(1)} \le \left(|Y|q^{-(\ell+1)/2}\right)^{\ell-j}
            (\log q)^{O(1)}
            \le
            \left(|Y|q^{-\ell/2}\right)^{\ell-j}(\log q)^{O(1)},
        \]
        where the second inequality follows from \ref{item:size-Yv1}. If $j=1$, the number of such edges is exactly $e(\calH_{v_1})$.
        Since
        \[
            (\ell-1)e(\calH_{v_1})
            =\sum_{v\in Y_{v_1}}d_{\calH_{v_1}}(v)
            \le |Y_{v_1}|\Delta_1(\calH_{v_1}),
        \]
        by \ref{item:codegrees-in-Yv1} together with \ref{item:size-Yv1}, we have
        \[
            e(\calH_{v_1})
            \le |Y_{v_1}|
            \left(|Y_{v_1}|q^{-(\ell-1)/2}\right)^{\ell-2}
            (\log q)^{O(1)}
            \le
            \left(|Y|q^{-\ell/2}\right)^{\ell-1}(\log q)^{O(1)}.
        \]
    
        Suppose now that $1\notin Q$. For $j\in[\ell-1]$, define
        \[
            a(j):=
            \begin{cases}
            (\ell-1)/2, & \text{if }j=\ell/2,\\
            \min\{j,\ell/2\}, & \text{otherwise}.
            \end{cases}
        \]
        Choose $Q'\subseteq Q$ with $|Q'|=\min\{j,\ell/2\}$, such that $\ell/2+1\notin Q'$ whenever $j>\ell/2$. If
        $e=\{v_1<\cdots<v_\ell\}\in E(\calH)$ satisfies $v_{q_t}=u_t$ for every $t \in [j]$, then
        \[
            v_1\in N^L(\{v_i:i\in Q'\})\cap Y.
        \]
        Hence, by \ref{item:neighbourhood-in-Y}, the number of possible choices for
        $v_1 \in R$ is at most $|Y|q^{-a(j)}(\log q)^{O(1)}$. For each fixed choice of $v_1 \in R$, the $j$ fixed vertices must all lie
        in an edge of $\calH_{v_1}$. Thus, by \ref{item:codegrees-in-Yv1}, the number
        of such edges is at most
        \[
            \left(|Y_{v_1}|q^{-(\ell-1)/2}\right)^{\ell-1-j}
            (\log q)^{O(1)}.
        \]
        Since $|Y_{v_1}|\le |Y|/(q(\log q)^8)$, multiplying the two bounds, the number of edges
        $e=\{v_1<\cdots<v_\ell\}\in E(\calH)$ satisfying
        $v_{q_t}=u_t$ for every $t\in[j]$ is at most
        \[
            |Y|q^{-a(j)}
            \left(|Y|q^{-(\ell+1)/2}\right)^{\ell-1-j}
            (\log q)^{O(1)}
            =
            |Y|^{\ell-j}
            q^{-a(j)-\frac{(\ell+1)(\ell-1-j)}2}
            (\log q)^{O(1)}.
        \]
        Thus, to prove the claim, it suffices to check that
        \[
            a(j)\ge\frac{j+1}{2}.
        \]
        If $j<\ell/2$, then $a(j)=j\ge(j+1)/2$. If $j>\ell/2$, then
        $a(j)=\ell/2\ge(j+1)/2$, since $j\le\ell-1$. Finally, if
        $j=\ell/2$, then
        $a(j)=(\ell-1)/2\ge(j+1)/2$ since $\ell\ge4$.
        This proves the claim.
    \end{proof}
    Combining Claims \ref{claim:edgecount2} and \ref{claim:codegrees2} completes the proof.
\end{proof}

\section{The number of $K_{\ell}$-free subsets}
\label{sec: hypergraph-containers}
In this section, we use the hypergraph container method to prove Lemma~\ref{lem:counting-kl-free}. We begin by stating the hypergraph container theorem.

\begin{theorem}[{\cite[Theorem~1.5]{morris-samotij-saxton}}] \label{thm:containers}
    Suppose that positive integers $b, \ell, \rho$ and a non-empty $\ell$-uniform hypergraph $\calH$ satisfy
    \[ \Delta_j(\calH) \le \left( \frac{b}{v(\calH)} \right)^{j-1} \frac{e(\calH)}{\rho}, \]
    for every $j \in [\ell]$. Then there exists a family $\cC$ of subsets of $V(\calH)$ such that every independent set in $\calH$ is contained in a member of $\cC$. Furthermore, for every $C \in \cC,$ we have $|C| \le v(\calH) - 2^{-\ell (\ell+1)} \rho,$ and
    \[ |\cC| \le \sum_{i=0}^{\ell b} \binom{v(\cH)}{i}. \]
\end{theorem}

\begin{lemma}
    \label{lem: containers for large sets}
    There exists a constant $C_3=C_3(\ell,s)$ such that the following holds for all sufficiently large $q$. Let $X \subseteq V(\Gamma)$ satisfy $|X| \ge q^{s+\ell-3}(\log q)^{C_3}$. Then there exists a family $\cC$ of subsets of $X$ such that every $K_{\ell}$-free subset of $X$ is contained in some member of $\cC$. Moreover, for every $C \in \cC$, we have $|C| \le (1-(\log q)^{-C_3})|X|$ and $|\cC| \le |X|^{\ell q^{\ell/2}}$.
\end{lemma}

\begin{proof}
    Let $C_3=C_3(\ell,s)$ be a sufficiently large constant to be chosen implicitly. Fix a subset $X\subseteq V(\Gamma)$ satisfying
    $|X|\ge q^{s+\ell-3}(\log q)^{C_3}$. By Corollary~\ref{cor:regularised-subset}, there exists a subset $Y\subseteq X$ with $|Y|\ge |X|/(\log q)^3$ such that $m_Y^Lm_Y^R\le |Y|(\log q)^7/q$. For $C_3$ sufficiently large, $|Y| \ge q^{s+\ell-3}(\log q)^{C_2}$, where $C_2=C_2(\ell,s)$ is given by Lemma~\ref{lem:final-hypergraph}. 

    Applying Lemma~\ref{lem:final-hypergraph} to $Y$, we obtain an $\ell$-uniform hypergraph $\calH$ on $Y$ such that every edge of $\calH$ forms an $\ell$-clique in $\Gamma[Y]$. Furthermore,
    $\Delta_j(\calH)\le (|Y|q^{-\ell/2})^{\ell-j}(\log q)^{O(1)}$ for every $j\in[\ell]$, and
    $e(\calH)\ge |Y|^\ell q^{-\binom{\ell}{2}}(\log q)^{-O(1)}$. Hence, for every $j\in[\ell]$,
    \[
    \frac{\Delta_j(\calH)}{e(\calH)}
    \le |Y|^{-j}q^{\frac{\ell}{2}(j-1)}(\log q)^{O(1)}
    \le \left(\frac{q^{\ell/2}}{|Y|}\right)^{j-1}
    \frac{(\log q)^{C_3/2}}{|Y|},
    \]

    for $C_3$ sufficiently large. Therefore, applying Theorem~\ref{thm:containers} with
    $b=q^{\ell/2}$ and $\rho=|Y|(\log q)^{-C_3/2}$, we obtain a family $\cC'$ of subsets of $Y$. Since every edge of $\calH$ forms a copy of $K_\ell$ in $\Gamma[Y]$, every $K_\ell$-free subset of $Y$ is an independent set in $\calH$, and hence is contained in some member of $\cC'$. Moreover, every $C'\in\cC'$ satisfies $|C'|\le |Y| - \rho / \log q =  (1-(\log q)^{-C_3/2-1})|Y|$. Finally, using $\sum_{i=0}^k\binom{n}{i}\le n^k$ for $n\ge k\ge2$, we have $|\cC'| \le |Y|^{\ell q^{\ell/2}}
    \le |X|^{\ell q^{\ell/2}}$.
    
    Now define $\cC:=\{(X\setminus Y)\cup C':C'\in\cC'\}$. Then every $K_\ell$-free subset of $X$ is contained in some member of $\cC$. Furthermore, for every $C=(X\setminus Y)\cup C'\in\cC$,
    \[
    |C|
    \le |X|-|Y|(\log q)^{-C_3/2-1}
    \le (1-(\log q)^{-C_3/2-4})|X|\le (1-(\log q)^{-C_3})|X| ,
    \]
    where we used $|Y|\ge |X|/(\log q)^3$, together with the fact that $C_3$ is sufficiently large. Also,
    $|\cC|=|\cC'|\le |X|^{\ell q^{\ell/2}}$.
    This proves the lemma.
\end{proof}

We now iterate Lemma~\ref{lem: containers for large sets} to obtain a family of containers covering all $K_\ell$-free subsets of $V(\Gamma)$.

\begin{lemma}
\label{lem:iterated-containers}
There exists a family $\cF$ of subsets of $V(\Gamma)$ such that every
$K_\ell$-free subset of $V(\Gamma)$ is contained in some member of $\cF$,
every $C\in\cF$ satisfies $|C|\le q^{s+\ell-3}(\log q)^{O(1)}$ and $|\cF|\le |V(\Gamma)|^{q^{\ell/2}(\log q)^{O(1)}}$.
\end{lemma}

\begin{proof}
Let $C_3=C_3(\ell,s)$ be the constant given by Lemma \ref{lem: containers for large sets}, and set $w:=q^{s+\ell-3}(\log q)^{C_3}$.
Let $\cF_0=\{V(\Gamma)\}$. For $i\ge0$, construct $\cF_{i+1}$ as follows. For each $X\in\cF_i$, if $|X|<w$, include $X$ in $\cF_{i+1}$. Otherwise, let $\cC_X$ be the family given by Lemma~\ref{lem: containers for large sets} applied to $X$, and include all members of $\cC_X$ in $\cF_{i+1}$.

Let $\eta=(\log q)^{-C_3}$. Then every member of $\cC_X$ has size at most $(1-\eta)|X|$. By induction on $i$, every $K_\ell$-free subset of $V(\Gamma)$ is contained in some member of $\cF_i$, and every $C\in\cF_i$ satisfies $|C|\le \max\left\{w,(1-\eta)^i|V(\Gamma)|\right\}$.
Since $|V(\Gamma)|=q^{O(1)}$, there exists a constant $C_4=C_4(\ell,s)$ such that $(1-\eta)^T|V(\Gamma)|<w$, where $T=(\log q)^{C_4}$. Hence every member of $\cF_T$ has size at most $w$.

Moreover, by Lemma~\ref{lem: containers for large sets}, every set in $\cF_i$ gives rise to at most $|V(\Gamma)|^{\ell q^{\ell/2}}$
members of $\cF_{i+1}$. Therefore, by induction, $|\cF_i|\le |V(\Gamma)|^{\ell q^{\ell/2}i}$ and in particular
\[
|\cF_T|
\le |V(\Gamma)|^{\ell q^{\ell/2}T}=|V(\Gamma)|^{\ell q^{\ell/2}(\log q)^{C_4}}
\le |V(\Gamma)|^{q^{\ell/2}(\log q)^{O(1)}}.
\]

Taking $\cF=\cF_T$ proves the lemma.
\end{proof}

Lemma \ref{lem:counting-kl-free} now follows easily.

\begin{proof}[Proof of Lemma \ref{lem:counting-kl-free}]
By Lemma~\ref{lem:iterated-containers}, there exists $C_5=C_5(\ell,s)$ and a family $\cF$ such that every $K_\ell$-free subset of $V(\Gamma)$ is contained in some $C\in\cF$, every $C\in\cF$ satisfies $|C|\le q^{s+\ell-3}(\log q)^{C_5}$, and $|\cF|\le |V(\Gamma)|^{q^{\ell/2}(\log q)^{C_5}}$. Choose $\beta$ sufficiently large compared to $C_5$ and set $m=q^{\ell/2}(\log q)^\beta$. Since $|V(\Gamma)|=q^{O(1)}$, for sufficiently large $q$ we have $|\cF|\le 2^m$. Moreover, for every $C\in\cF$,
    \[
        \binom{|C|}{m}
        \le \left(\frac{e|C|}{m}\right)^m
        \le \left(e q^{s+\ell/2-3}(\log q)^{C_5-\beta}\right)^m
        \le \left(\frac{q^{s+\ell/2-3}}2\right)^m.
    \]
    Hence the number of $K_\ell$-free subsets of size $m$ is at most
    \[
        |\cF|\max_{C\in\cF}\binom{|C|}{m}
        \le 2^m\left(\frac{q^{s+\ell/2-3}}2\right)^m
        =\left(q^{s+\ell/2-3}\right)^m,
    \]
    as required.
\end{proof}

\paragraph{Acknowledgement.} We thank Levente Bodn\'ar for useful discussions on the content of the appendix.

\paragraph{Statement on AI use.} The ideas in this paper were found without AI assistance. We used ChatGPT Pro to help us with the writing.

\appendix

\section{Improved lower bounds for the Erd\H{o}s-Rogers function}
\label{app:lower-bound}

As mentioned in the introduction, we will discuss in this appendix how Sudakov's approach in \cite{Sudakov2005Large} may be refined to obtain improved lower bounds for $f_{\ell,s}(n)$, whenever $\ell \ge 4$ and $s \ge \ell+3$. To prove a lower bound on $f_{\ell,s}(n)$, one needs to show that every $K_s$-free graph on $n$ vertices contains a large $K_{\ell}$-free induced subgraph. To this end, we will need the following well-known Tur\'{a}n-type estimate on the independence number of a $k$-uniform hypergraph, also used in \cite{Sudakov2005New, Sudakov2005Large}.

\begin{lemma}
\label{lem: ind-set-bound}
    Let $k \ge 2$, and let $\calG$ be a $k$-uniform hypergraph with $n$ vertices and $m$ edges. Then $\calG$ contains an independent set of size at least 
    \[
    \Omega\!\left(
        \min\left\{
            n,\,
            n\left(\frac{n}{m}\right)^{\!1/(k-1)}
        \right\}
    \right).
    \]
\end{lemma}

We will focus on the case $s=\ell+3$. The approach
in~\cite{Sudakov2005Large} is to bound the number of triangles
in a $K_{\ell+3}$-free graph. This is done by applying dependent
random choice to bound the number of edges in a vertex
neighbourhood in terms of the size of the largest $K_\ell$-free
induced subgraph, independently of the size of the neighbourhood.
We retain the dependence on the neighbourhood size, obtaining
an upper bound on the number of triangles in terms of the
degrees.

Our additional step is to apply dependent random choice to the
whole graph and count triangles in the resulting common
neighbourhood, rather than edges as in Sudakov's treatment of
the $(\ell,\ell+2)$ case. This yields an upper bound on the
degrees in terms of the triangle count. Combining the
two estimates gives a stronger bound on the number of triangles, resulting in an improved lower bound on $f_{\ell,\ell+3}(n)$. Concretely, we prove the following result.

\begin{theorem}
    \label{thm: improved-lower-bound}
    For every $\ell \ge 3$, we have 
    \[
    f_{\ell,\ell+3}(n) \ge \Omega(n^{\frac{23\ell-39}{46\ell-45}}).
    \]
\end{theorem}

Sudakov's recursive lower bound in \cite{Sudakov2005Large} gives $f_{\ell,\ell+3}(n) \ge \Omega(n^{\frac{(\ell-1)(3\ell-4)}{6\ell^2-9\ell+4}})$. It is not hard to see that for $\ell \ge 4$, Theorem~\ref{thm: improved-lower-bound} yields a polynomial improvement over Sudakov's lower bound. Before we turn to the proof, we remark that, for larger values of $\ell$, the bound in
Theorem~\ref{thm: improved-lower-bound} can be further improved
by optimising the number of sampled vertices in the two
dependent random choice steps, but we omit the details.
Finally, Sudakov's recurrence~\cite[Theorem~3.1]{Sudakov2005Large}
implies that, for each fixed $\ell\ge4$, our improvement for
$f_{\ell,\ell+3}(n)$ also yields a polynomial improvement over
his lower bound for $f_{\ell,s}(n)$ whenever $s>\ell+3$.

We now proceed with the proof of Theorem \ref{thm: improved-lower-bound}.

\begin{proof}
    Let $G$ be a $K_{\ell+3}$-free graph on $n$ vertices, and let $h$ denote the size of the largest $K_\ell$-free induced subgraph of $G$. Observe that since $G$ is $K_{\ell+3}$-free, the common neighbourhood of a triangle is a $K_{\ell}$-free induced subgraph, and therefore has size at most $h$. Let $T$ denote the number of triangles in $G$, and define $D= \sum_{v \in V(G)}d_v^{2}/n$, where $d_v$ is the degree of the vertex $v$.
        
    Let $v_1, v_2, v_3$ be vertices of $G$ (not necessarily distinct) chosen uniformly and independently at random. Let $X$ and $Y$ denote respectively the number of vertices and the number of triangles in their common neighbourhood. Since the common neighbourhood of any triangle has size at most $h$, we have $\mathbb{E}[Y] \le T(h/n)^3$. On the other hand, 
    \[
    \mathbb{E}[X]=\sum_{v \in V(G)}\frac{d_v^3}{n^3} \ge \frac{D^{3/2}}{n^2},
    \]
    where the inequality uses H\"{o}lder's inequality applied with functions $f(v) = d_v^2, g(v) = 1$ and conjugate exponents $3/2$ and $3$.
    
    Consequently, letting $Z=X-\frac{nD^{3/2}}{2Th^3}Y-\frac{D^{3/2}}{2n^2}$, and using the linearity of expectation, we have $\mathbb{E}[Z] \ge 0$. Fixing a choice of $v_1, v_2, v_3$ such that $Z \ge 0$, it follows that $X \ge \frac{D^{3/2}}{2n^2}$ and $\frac{X}{Y} \ge \frac{nD^{3/2}}{2Th^3}$. Applying Lemma~\ref{lem: ind-set-bound} to the $3$-uniform hypergraph whose vertices are the vertices in the common neighbourhood of $v_1, v_2, v_3$ and whose hyperedges are the triangles contained in this neighbourhood, we obtain a triangle-free (and hence $K_{\ell}$-free, since $\ell \ge 3$) induced subgraph of size at least 
    \[
    \Omega \left( \min \left\{ X,X \left( \frac{X}{Y}\right)^{1/2}\right\}\right) \ge \Omega \left( \min \left\{ \frac{D^{3/2}}{n^2}, \frac{D^{9/4}}{n^{3/2}T^{1/2}h^{3/2}}\right\}\right). 
    \]
    
    Since $h$ is the size of the largest $K_\ell$-free induced subgraph of $G$, the expression above must be at most $h$, which gives 
    \begin{equation}
    \label{upper-bound-D}
      D \le O(\max\{n^{4/3}h^{2/3}, n^{2/3}h^{10/9}T^{2/9}\})
    \end{equation}
    For a vertex $v$, let $T_v$ denote the number of triangles in $G$ containing $v$. Then, the subgraph induced by the neighbourhood of $v$, denoted by $G_v$, has $d_v$ vertices and $T_v$ edges. Moreover, since $G$ is $K_{\ell+3}$-free, $G_v$ is $K_{\ell+2}$-free. Consequently, every edge in $G_v$ has at most $h$ common neighbours in $G_v$. Applying Sudakov's dependent random choice argument, Lemma~2.1 in \cite{Sudakov2005Large} with $i=2$, we obtain an induced subgraph of $G_v$ with at least $2T_v^2/d_v^3$ vertices and average degree at most $h^2d_v/T_v$. Applying Lemma~\ref{lem: ind-set-bound} to this induced subgraph, there exists an independent set (and hence a $K_{\ell}$-free induced subgraph of $G$) of size at least 
    \[
    \Omega \left( \min \left\{ \frac{T_v^2}{d_v^3}, \frac{T_v^3}{h^2d_v^4}\right\}\right).
    \]
    Since the above expression is at most $h$, we get 
    \[
    T_v \le O(\max\{h^{1/2}d_v^{3/2}, hd_v^{4/3}\}) \le O(h^{1/2}d_v^{3/2}+hd_v^{4/3}).
    \]
    Summing over all vertices $v \in V(G)$, we get 
    \[
    T \le O\left(h^{1/2}\sum_{v \in V(G)}d_v^{3/2}+h \sum_{v \in V(G)}d_v^{4/3}\right) \le O(nh^{1/2}D^{3/4}+nhD^{2/3}).
    \]
    This implies that $T \le O(\max \{nh^{1/2}D^{3/4}, nhD^{2/3}\})$. Plugging in the upper bound on  $D$ from~\eqref{upper-bound-D}, we get
    \begin{equation*}
    T \le O(\max \{n^2h, n^{3/2}h^{4/3}T^{1/6},n^{17/9}h^{13/9},n^{13/9}h^{47/27}T^{4/27}\}).
    \end{equation*}
    This yields
    \begin{equation}
    \label{eqn: upper-bound-T}
      T \le O(\max \{n^2h, n^{9/5}h^{8/5},n^{17/9}h^{13/9},n^{39/23}h^{47/23}\}).  
    \end{equation}
    As we previously observed, the common neighbourhood of any triangle in $G$ has size at most $h$, and therefore each triangle is contained in at most $h^{\ell-3}$ copies of $K_\ell$. Hence, the number of copies of $K_{\ell}$ in $G$ is at most $Th^{\ell-3}$. Applying Lemma~\ref{lem: ind-set-bound} to the hypergraph whose vertices are the vertices of $G$ and whose edges are the copies of $K_{\ell}$ in $G$, we obtain a $K_{\ell}$-free induced subgraph of $G$ of size at least
    \[
    \Omega\!\left(
        \min\left\{
            n,\,
            n\left(\frac{n}{Th^{\ell-3}}\right)^{\!1/(\ell-1)}
        \right\}
    \right).
    \]
    We may assume that the second term in the expression above is smaller, as otherwise we obtain a $K_{\ell}$-free subgraph of size $\Omega(n)$, and the proof is complete. Thus, we have
    \begin{equation}
    \label{eqn: lower-bound-T}
    T \ge \Omega \left( \frac{n^{\ell}}{h^{2\ell-4}}\right).
    \end{equation}
    Combining with~\eqref{eqn: upper-bound-T}, we get
    \[
    \frac{n^{\ell}}{h^{2\ell-4}} \le O(\max \{n^2h, n^{9/5}h^{8/5},n^{17/9}h^{13/9},n^{39/23}h^{47/23}\}).
    \]
    This implies that 
    \[
    h \ge \Omega \left(\min\{n^{\frac{\ell-2}{2\ell-3}}, n^{\frac{5\ell-9}{10\ell-12}}, n^{\frac{9\ell-17}{18\ell-23}}, n^{\frac{23\ell-39}{46\ell-45}}\}\right).
    \]
    It is easy to check that for $\ell \ge 3$, the last term in the expression above is the smallest, finishing the proof.  
\end{proof}

\end{document}